\documentclass[a4paper, 10pt, oneside, onecolumn]{article}

\RequirePackage{geometry}
\usepackage[utf8]{inputenc}
\usepackage[T1]{fontenc}
\usepackage{lmodern}
\usepackage{amsmath}
\usepackage{amssymb}
\usepackage{amsfonts}
\usepackage{amsthm}
\usepackage{xcolor}
\usepackage{hyperref}
\usepackage{cleveref}
\usepackage{mathtools}
\usepackage{constants}
\usepackage[autostyle=true]{csquotes}

\newconstantfamily{C}{symbol=C}

\hypersetup{
	colorlinks=false,
	pdfborder={0 0 0},
	pdftitle={Does strong repulsion lead to smooth solutions in a repulsion-attraction chemotaxis system even when starting with highly irregular initial data?},
	pdfauthor={Frederic Heihoff},
	pdfkeywords={},
	bookmarksopen=true,
}

\renewcommand{\phi}{\varphi}

\newtheorem{base}{Base}[section]
\numberwithin{equation}{section}

\theoremstyle{plain}
\newtheorem{theorem}[base]{Theorem}
\newtheorem{lemma}[base]{Lemma}

\newtheorem{corollary}[base]{Corollary}

\theoremstyle{definition}

\newcommand{\R}{\mathbb{R}}
\newcommand{\N}{\mathbb{N}}
\renewcommand{\d}{\,\mathrm{d}}
\newcommand{\laplace}{\Delta}
\newcommand{\grad}{\nabla}
\renewcommand{\div}{\nabla \cdot}
\renewcommand{\L}[1]{{L^{#1}(\Omega)}}

\newcommand{\defs}{\coloneqq}
\newcommand{\sfed}{\eqqcolon}
\newcommand{\stext}[1]{\;\;\text{ #1 }\;\;}
\newcommand{\eps}{\varepsilon}

\newcommand{\tmax}{{T_{\mathrm{max}}}}
\newcommand{\tmaxeps}{T_{\mathrm{max}, \eps}}

\newcommand{\ue}{u_\eps}
\newcommand{\ve}{v_\eps}
\newcommand{\we}{w_\eps}
\newcommand{\ze}{z_\eps}
\newcommand{\uet}{u_{\eps t}}
\newcommand{\vet}{v_{\eps t}}

\newcommand{\Mp}{\mathcal{M}_+(\overline{\Omega})}

\newcommand{\uparam}{{\kappa}}

\newcommand{\Capprox}{\widehat{C_{\mathrm{S2}}}}
\newif\ifclarification
\clarificationtrue

\allowdisplaybreaks

\newcommand\numberthis{\addtocounter{equation}{1}\tag{\theequation}}

\makeatletter
\g@addto@macro\bfseries{\boldmath}
\makeatother

\title{Does strong repulsion lead to smooth solutions in a repulsion-attraction chemotaxis system even when starting with highly irregular initial data?}
\author{
	Frederic Heihoff\footnote{fheihoff@math.uni-paderborn.de}\\
	{\small Institut f\"ur Mathematik, Universit\"at Paderborn,}\\
	{\small 33098 Paderborn, Germany}
}
\date{}
\begin{document}

\maketitle

\begin{abstract}
	\noindent
	It has been well established that, in attraction-repulsion Keller--Segel systems of the form
	\begin{equation*}
		\left\{
		\begin{aligned}
			u_t &= \Delta u - \chi \nabla \cdot (u\nabla v) + \xi \nabla \cdot (u\nabla w),\\
			\tau v_t &= \Delta v + \alpha u - \beta v, \\
			\tau w_t &= \Delta w + \gamma u - \delta w
		\end{aligned}
		\right.
	\end{equation*}
	in a smooth bounded domain $\Omega \subseteq \mathbb{R}^n$, $n\in\mathbb{N}$, with Neumann boundary conditions and parameters $\chi, \xi \geq 0$, $\alpha,\beta,\gamma,\delta > 0$ and $\tau \in \{0,1\}$, finite-time blow-up can be ruled out in many scenarios given sufficiently smooth initial data if the repulsive chemotaxis is sufficiently stronger than its attractive counterpart. In this paper, we will go - in a sense - a step further than this by studying the same system with initial data that could already be understood as being in a blown-up state (e.g.\ a positive Radon measure for the first solution component) and then ask the question whether sufficiently strong repulsion has enough of a regularizing effect to lead to the existence of a smooth solution, which is still connected to said initial data in a sensible fashion. Regarding this, we in fact establish that the construction of such a solution is possible in the two-dimensional parabolic-parabolic system and the two- and three-dimensional parabolic-elliptic system under appropriate assumptions on the interaction of repulsion and attraction as well as the initial data. 
	\\[0.5em]
	\textbf{Keywords:} attraction-repulsion; Keller--Segel; measure-valued initial data; smooth solution \\
	\textbf{MSC 2020:} 35Q92 (primary); 35K10;	35J15; 35K55; 35A09; 35B65; 92C17
\end{abstract}

\section{Introduction}
One of the fundamental questions in the mathematical analysis of partial differential equation models of chemotaxis (cf.\ \cite{BellomoMathematicalTheoryKeller2015}) used in the analysis of various biological organisms is whether the struggle between the forces of diffusion and chemotaxis ultimately leads to finite-time blow-up, regularization or even stabilization of solutions. Importantly, the interest in the detailed analysis of solution behavior of such models does not purely stem from mathematical curiosity about these often challenging systems, but also chiefly from a desire to see whether the models appropriately represent the observed behavior of the biological organisms that inspired them and thus could prove to be a useful tool in practical applications. 
\\[0.5em]
Most of the early analytical endeavors in this area, growing from the seminal paper by Keller and Segel in 1970 (cf.\ \cite{KellerInitiationSlimeMold1970}) and emboldened by the subsequent successful analysis of their proposed model (cf.\ \cite{BellomoMathematicalTheoryKeller2015}), have to our knowledge focused on the interplay of the regularizing effects of diffusion and the potentially destabilizing effects of attractive chemotaxis in various scenarios. Though more recently in an effort to e.g.\ understand Alzheimer's disease (cf.\ \cite{luca2003chemotactic}) or describe quorum-sensing effects observed in certain kinds of bacteria (cf.\ \cite{PainterVolumefillingQuorumsensingModels2002}), variations of the original Keller--Segel model have been introduced, which not only feature an attractive but also a repulsive chemical. A model of exactly this type will be the central object of interest in this paper. More precisely, we will be analyzing the system of partial differential equations
\begin{equation}\label{problem}
	\left\{
	\begin{aligned}
		u_t &= \laplace u - \chi \div (u\grad v) + \xi \div (u\grad w)&& \text{ on } \Omega \times (0,\infty), \\
		\tau v_t &= \laplace v + \alpha u - \beta v && \text{ on } \Omega \times (0,\infty), \\
		\tau w_t &= \laplace w + \gamma u - \delta w  && \text{ on } \Omega \times (0,\infty),\\
		0 &= \grad u \cdot \nu = \grad v \cdot \nu = \grad w \cdot \nu && \text{ on } \partial \Omega \times (0,\infty)
	\end{aligned}
	\right.
\end{equation}
in a smooth bounded domain $\Omega \subseteq \R^n$, $n\in\N$, with parameters $\chi,\xi \geq 0$, $\alpha,\beta,\gamma,\delta > 0$ and $\tau \in \{0,1\}$. Herein, the first equation models the movement of the organisms in question toward the attractive chemical represented by $v$ at rate $\chi$ and away from the repulsive chemical represented by $w$ at rate $\xi$. As seen in the second and third equation, both chemicals are produced by the organisms at rates depending on the choice of $\alpha$ and $\gamma$, respectively, and decay over time at rates $\beta$ and $\delta$, respectively. Notably, the second and third equation can either be of parabolic or elliptic type depending on the choice of parameter $\tau$. This choice is generally interpreted as either the chemicals conforming to a time evolution at similar time scales as the organisms in the parabolic case or the chemicals reacting almost immediately to changes in organism concentration in the elliptic case.
\paragraph{Prior work.}From an intuitive standpoint and in many cases very much by design, one would expect a sufficiently strong repulsive influence to counteract the aggregation behavior often underlying finite-time blow-up and thus leading more readily to the global existence of classical solutions. In fact given regular initial data, this intuition seems to be supported by prior mathematical analysis as it has been shown that, if the repulsive taxis in the system (\ref{problem}) is strictly stronger than its attractive counterpart in the sense that $\xi \gamma - \chi \alpha > 0$, then global classical solutions exist in two dimensions if $\tau = 1$ and in arbitrary dimension if $\tau = 0$ (cf.\ \cite{LiuGlobalBoundednessFully2015}, \cite{TaoCOMPETINGEFFECTSATTRACTION2013}). It has further been shown that under potentially additional parameter restrictions said solutions are even globally bounded or exhibit certain large-time stabilization properties (cf.\ \cite{JinBoundednessAttractionrepulsionKellerSegel2015}, \cite{JinGlobalStabilizationFull2020}, \cite{LinGlobalExistenceConvergence2016}, \cite{LiuPatternFormationAttractionrepulsion2013}).  Conversely if attraction dominates over repulsion, the finite-time blow-up results already established for attraction-only systems (cf.\ e.g.\ \cite{NagaiblowupRadiallySymmetric1995}, \cite{NagaiblowupNonradialSolutions2001}, \cite{WinklerAggregationVsGlobal2010}, \cite{WinklerFinitetimeblowupHigherdimensional2013}) seem to largely translate to the competition case (cf.\ \cite{JinBoundednessblowupCritical2016}, \cite{lankeitFinitetimeblowupThreedimensional2021}, \cite{TaoCOMPETINGEFFECTSATTRACTION2013}). Naturally apart from system (\ref{problem}), which stays fairly close to the original Keller--Segel system, many of its canonical variations have also been explored as well. To mention a few, there has been some consideration of models, in which the attractant is consumed instead of produced (cf.\ \cite{FrassuBoundednessNonlinearAttractionrepulsion2021}, \cite{FrassuBoundednessChemotaxisSystem2021}), in which the taxis mechanisms further interact with a logistic source term (cf.\ \cite{LiAttractionrepulsionChemotaxisSystem2016}, \cite{ZhangAttractionrepulsionChemotaxisSystem2016}), or in which the movement mechanisms feature some form of nonlinearity (cf.\ \cite{ChiyoBoundednessFullyParabolic2022}, \cite{FrassuBoundednessNonlinearAttractionrepulsion2021}, \cite{LinBoundednessBlowHigherdimensional2016}). Apart from this, there has also been some analysis of the interaction between attraction and consumption in the whole space case (cf.\ \cite{NagaiGlobalExistenceSolutions2021}, \cite{YamadaGlobalExistenceBoundedness2022}).
\\[0.5em]
Let us further briefly mention that there is another prominent setting which at times deals with interaction of attraction and repulsion, namely predator-prey models. Here, the predators, which are generally modeled by the first of two diffusive equations, are attracted by the prey. The prey in turn, which is modeled by the second of the aforementioned equations, is repelled by the predators. If both taxis mechanisms are present in such a setting however, the situation seems to be much less clear cut than in the attraction-repulsion model (\ref{problem}) as even the construction of (potentially only generalized) solutions seems to be rather challenging given that cross-diffusion is present in both equations (cf.\ \cite{FuestGlobalSolutionsHomogeneous2020}, \cite{FuestGlobalWeakSolutions2021}, \cite{TaoExistenceTheoryQualitative2020}). As such, most efforts in this area focus only on one of the two mechanisms and remove the other.
\paragraph{Main result.} 
As the boundary between finite-time blow-up and global existence for the system (\ref{problem}) has at this point been fairly thoroughly explored in the literature discussed above, we will in this paper go a step further in a sense by addressing the following question: If the system (\ref{problem}) already starts in a state resembling blow-up (e.g.\ a Dirac measure for the first solution component), can sufficiently strong repulsion be enough of a regularizing influence counteracting  attraction to still yield classical solutions, which remain connected to the initial state in a reasonable way? Notably, not only is it well-documented that such an immediate smoothing property is exhibited by the pure Neumann heat equation (cf.\ \cite{HenryGeometricTheorySemilinear1981}, \cite{LunardiAnalyticSemigroupsOptimal1995}), but very similar questions have been posed and positively answered for other cross-diffusion systems, providing us with a degree of optimism regarding this type of inquiry. To be more explicit, construction of such solutions has been accomplished in the parabolic-elliptic repulsive Keller--Segel system on two- and three-dimensional domains (cf.\ \cite{HeihoffExistenceGlobalSmooth2021}) as well as in the standard parabolic-parabolic attractive Keller--Segel system on one-dimensional domains (cf.\ \cite{WinklerInstantaneousRegularizationDistributions2019}) and on two-dimensional domains under the regularizing influence of a logistic source term (cf.\ \cite{LankeitIrregularInitialData}). Without the regularizing aid of such a logistic source, mild solutions, which become immediately bounded in $L^{\frac{3}{2}}(\Omega)$, have also been constructed for the attractive system on two-dimensional domains for measure-valued initial data with small mass or any $L^1(\Omega)$ initial data (cf.\ \cite{BilerLocalGlobalSolvability1998}).
There have also been some similar discussions in the two-dimensional whole space (cf.\ \cite{BedrossianExistenceUniquenessLipschitz2014}, \cite{BilerCauchyProblemSelfsimilar1995}, \cite{BilerExistenceSolutionsKellerSegel2015}, \cite{RaczynskiStabilityPropertyTwodimensional2009}), the radially symmetric (cf.\ \cite{WangImmediateRegularizationMeasuretype2021}, \cite{WinklerHowStrongSingularities2019}) as well as the toroidal case (cf.\ \cite{SenbaWeakSolutionsParabolicelliptic2002}).
\\[0.5em] 
As such in an effort to expand on this exact area of inquiry and in many ways as a direct extension of the results presented in \cite{HeihoffExistenceGlobalSmooth2021}, \cite{LiuGlobalBoundednessFully2015} as well as \cite{TaoCOMPETINGEFFECTSATTRACTION2013},
the main result of this paper is the construction of classical solutions to (\ref{problem}) with measure-valued initial data for the first solution component and, in the parabolic-parabolic case, appropriately regular initial data for the second and third solution component under sufficiently strong repulsion in the two-dimensional parabolic-parabolic and two- as well as three-dimensional parabolic-elliptic case. More precisely, we will prove the following result:
\begin{theorem}\label{theorem:main}
	Let $\Omega \subseteq \R^n$, $n \in \N$, be a bounded domain with a smooth boundary, $\chi, \xi \geq 0$ and $\alpha, \beta, \gamma, \delta > 0$ as well as $\tau \in \{0,1\}$. Let further $u_0 \in \Mp$ be an initial datum with $m \defs u_0(\overline{\Omega}) > 0$, where $\Mp$ is the set of positive Radon measures with the vague topology. If $\tau = 1$, let further $v_0, w_0 \in W^{1,r}(\Omega)$ with $r \in (\tfrac{6}{5},2)$ be some additional nonnegative initial data.
	If
	\begin{alignat}{3}
		\tau &= 1, \quad n &= 2 \stext{ and } \xi\gamma  - \chi\alpha  &\geq 0 \text{ or } \tag{S1}\label{scenario_1} \\
		\tau &= 0, \quad n &= 2 \stext{ and } \xi\gamma  - \chi\alpha  &\geq -\frac{C_{\mathrm{S2}}}{m} \text{ or }\tag{S2}\label{scenario_2} \\
		\tau &= 0, \quad n &= 3 \stext{ and } \xi\gamma  - \chi\alpha  &\geq 0 \stext{ as well as } u_0 \in L^\uparam(\Omega) \text{ with some } \uparam \in (1,2), \tag{S3}\label{scenario_3}
	\end{alignat}
	where $C_{\mathrm{S2}} > 0$ is a constant only depending on the domain $\Omega$, then there exist nonnegative functions $u \in C^{2,1}(\overline{\Omega}\times(0,\infty))$ and $v,w \in C^{2,\tau}(\overline{\Omega}\times(0,\infty))$ solving (\ref{problem}) classically and attaining their initial data in the following fashion:
	\begin{align}
		u(\cdot, t) &\rightarrow u_0 &&\text{ in } \Mp \label{u_continuity},\\
		v(\cdot, t) &\rightarrow v_0 &&\text{ in } W^{1,r}(\Omega) \text{ if } \tau = 1 \label{v_continuity}, \\
		w(\cdot, t) &\rightarrow w_0 &&\text{ in } W^{1,r}(\Omega) \text{ if } \tau = 1 \label{w_continuity} 
	\end{align}
	as $t \searrow 0$, where we interpret the functions $u(\cdot, t)$, $t > 0$, as the positive Radon measures $ u(x,t)\mathrm{d} x$ with $\mathrm{d} x$ being the standard Lebesgue measure on $\Omega$.
\end{theorem}

\paragraph{Approach.} 
As is not uncommon, the construction of our desired solution will be based on approximating them by a family of solutions $(\ue, \ve, \we)_{\eps \in (0,1)}$, for which global existence is much easier to establish. To this end, we will spend the next section approximating our initial data by smooth functions in a fashion convenient for later arguments and then prove that with such smooth initial data classical solutions to (\ref{problem}) exist globally as an extension of the arguments presented in \cite{LiuGlobalBoundednessFully2015} and \cite{TaoCOMPETINGEFFECTSATTRACTION2013}. Additionally in this section, we also introduce the functions $\ze \defs \xi \we - \chi \ve$ for each $\eps \in (0,1)$, which allow us to transform the system (\ref{problem}) to the system (\ref{problem_simplified}), because the second system will prove more convenient for some of the later arguments. The remainder of this paper will then be devoted to deriving uniform a priori bounds for exactly these approximate solutions to facilitate a central compact embedding argument, which will serve as the source of our actual solutions, as well as to ensure that the thus constructed solutions have our desired continuity properties at $t=0$.
\\[0.5em]
Naturally any substantial a priori estimates have to necessarily decay toward $t = 0$ if they are to be uniform in the approximation parameter $\eps$ as our initial data are very irregular. As such, the a priori estimates serving as the linchpins of all further considerations will roughly have the form 
\begin{equation}\label{eq:linchpin_estimate1}
	G(\ue(\cdot, t),\ze(\cdot, t)) \leq C t^{-\lambda}	
\end{equation}
or 
\begin{equation}\label{eq:linchpin_estimate2}
	\int_0^t s^\lambda \, G(\ue(\cdot, s),\ze(\cdot, s)) \d s \leq C
\end{equation}
for some $\lambda > 0$, where $G$ is some key norm or functional and the parameter $\lambda$ represents how severely the estimate decays close to $t = 0$. The derivation of these estimates will take place in Section 3 and, while most of the arguments afterward will handle both the parabolic-parabolic and parabolic-elliptic cases in a fairly integrated fashion, said derivation will use very different methods depending on the value of $\tau$. 
\\[0.5em]
For the parabolic-elliptic case, we essentially adapt methods found in \cite{HeihoffExistenceGlobalSmooth2021} to the system (\ref{problem}). At its core, the argument boils down to testing the first equation in (\ref{problem}) with $\ue^{p-1}$ and then after two partial integrations directly replacing the resulting $\laplace \ve$ and $\laplace \we$ terms by lower order expressions using the elliptic second and third equations in (\ref{problem}). Applying carefully chosen interpolation inequalities as well as elliptic regularity theory to this then allows us to derive a differential inequality with super-linear decay for $\int_\Omega \ue^{p}$, which is sufficient to yield an estimate of the type (\ref{eq:linchpin_estimate1}) for $\int_\Omega \ue^p$ with any $p\in (1,\infty)$.
\\[0.5em]
In the parabolic-parabolic case, our approach hinges on the use of the well-known energy-type functional $\mathcal{F}(t) \defs \zeta\int_\Omega \ue \ln (\ue) + \frac{1}{2} \int_\Omega |\grad \ze|^2$ with $\zeta \defs \xi \gamma - \chi \alpha$. But as this functional cannot necessarily be uniformly bounded at $t = 0$ due to our initial data potentially not having finite energy, we further decouple it from the initial data by multiplying it with $t^{\lambda}$, $\lambda > 0$. Using testing based methods, analysis of this functional then not only yields an estimate of type (\ref{eq:linchpin_estimate1}) for the functional itself but crucially also a bound of type (\ref{eq:linchpin_estimate2}) for the dissipative terms $\int_\Omega |\laplace \ze|^2$ and $\zeta \int_\Omega |\grad \sqrt{\ue}|^2$. Notably while the first set of bounds could also be achieved by a similar super-linear decay approach as described in the previous paragraph, the latter bounds seem to be much more conveniently accessible by analyzing the aforementioned time-dampened version of the functional $\mathcal{F}$. And importantly, it is in fact exactly said latter bounds that will allow us to prove our desired continuity properties at $t=0$ in this scenario, as well as allow us to derive a set of uniform bounds for $\int_\Omega \ue^2$ away from $t=0$ to give us a similar starting point to the parabolic-elliptic case for the next section.
\\[0.5em]
In Section 4, we then use the uniform bounds for $\int_\Omega \ue^n$ away from $t=0$, which we have at this point established in all scenarios, as the basis for a bootstrap argument taking us all the way to uniform bounds for the first solution components in $C^{2+ \theta, 1+\frac{\theta}{2}}(\overline{\Omega}\times[t_0,t_1])$ and uniform bounds for the second and third solution components in $C^{2+ \theta, \tau+\frac{\theta}{2}}(\overline{\Omega}\times[t_0,t_1])$ with $t_1 > t_0 > 0$. We do this mostly using the variation-of-constants representation of the involved equations or corresponding elliptic regularity theory as well as fairly standard Hölder regularity theory from \cite{LadyzenskajaLinearQuasilinearEquations1988}, \cite{LiebermanHolderContinuityGradient1987} as well as \cite{PorzioVespriHoelder}. Due to the compact embedding properties of Hölder spaces this immediately allows us to construct our desired solutions $(u,v,w)$ as limits of the thus far discussed approximate ones and argue that they classically solve (\ref{problem}) as the resulting strong convergence properties safely transfer any solution properties from the approximate solutions to their limits.
\\[0.5em]
It thus only remains to be shown that said solutions are connected to our initial data in the fashion outlined in (\ref{u_continuity})--(\ref{w_continuity}), which will be the main subject of Section 5. To do this for the first solution component, we essentially start by proving that the approximate solutions are uniformly continuous at $t=0$ in the sense of (\ref{u_continuity}). We accomplish this by showing that the space-time integral $\int_0^t \|\ue(\cdot,s) \grad \ze(\cdot, s)\|_\L{1} \d s$ related to the chemotaxis mechanism becomes uniformly small as $t$ goes to zero in all scenarios. That it is possible to prove such a property chiefly depends on the degradation toward zero of our estimates derived in Section 2 to be sufficiently benign. Notably, the magnitude of the degradation parameters naturally depends on the regularity of our initial data in the sense that better regularity leads to smaller degradation toward zero and thus the derivation of the aforementioned bound is, however indirect, the source of our initial data regularity needs in \Cref{theorem:main}. We then use said property combined with the first equation in (\ref{problem}) and the fundamental theorem of calculus to gain our desired uniform continuity property, which by virtue of the already established convergence properties translates immediately to our actual solutions. Using a convenient property of our initial data approximation, a similar argument built on semigroup methods grants us properties (\ref{v_continuity}) and (\ref{w_continuity}) in the parabolic-parabolic case. 

\section{Regularized initial data and approximate solutions}

From here on out, we fix the system parameters $\chi, \xi \geq 0$, $\alpha, \beta, \gamma, \delta > 0$ and $\tau \in \{0,1\}$ as well as a domain $\Omega \subseteq \R^n$, $n\in\N$, with a smooth boundary for the remainder of the paper. We further fix some initial data $u_0 \in \Mp$ with $m \defs u_0(\overline{\Omega}) > 0$, where $\Mp$ is the set of positive Radon measures with the vague topology, as well as nonnegative $v_0,w_0 \in W^{1,r}(\Omega)$ with $r \in (\frac{6}{5}, 2)$ if $\tau = 1$. Moreover if $u_0$ is additionally assumed to be an element of $L^\uparam(\Omega)$ for some $\uparam\in(1,2)$, we also fix this parameter $\uparam$. Otherwise, we let $\uparam$ be equal to $2$ so $\uparam$ is defined in all scenarios for convenience of notation in some later arguments.
\\[0.5em]
To construct the solutions laid out in \Cref{theorem:main}, we will use a family of approximate solutions, which will later be argued to converge to our desired solutions. This section will thus be devoted to the construction of said approximate solutions and to facilitate this we will begin by approximating our potentially highly irregular initial data by smooth functions, which will in fact be the only regularization necessary. 
\\[0.5em] 
As such, we now fix a family of approximate positive initial data $(u_{0,\eps})_{\eps \in (0,1)} \subseteq C^{\infty}(\overline{\Omega})$ such that 
\begin{equation}\label{eq:approx_u0_properties}
	u_{0,\eps} \rightarrow u_0 \;\;\;\; \text{ in } \Mp \text{ as } \eps\searrow 0 \stext{ as well as } \int_\Omega	u_{0,\eps} = u_0(\overline{\Omega}) = m \;\;\;\;\text{ for all } \eps \in (0,1),
\end{equation}
where we interpret the functions $u_{0,\eps}$ as the positive Radon measures $u_{0,\eps}(x)\mathrm{d}x$ with $\mathrm{d}x$ being the standard Lebesgue measure on $\Omega$. For a discussion of how this can be achieved, we refer the reader to e.g.\ \cite[Remark 2.2]{HeihoffExistenceGlobalSmooth2021}.
If $u_0$ is additionally an element of $L^\uparam(\Omega)$, we let $u_{0,\eps} \defs e^{\eps\laplace}u_0 \in C^{\infty}(\overline{\Omega})$ for all $\eps \in (0,1)$ instead, where $(e^{t\laplace})_{t \geq 0}$ is the Neumann heat semigroup on $\Omega$. By the continuity, positivity and mass conservation properties of said semigroup, this not only directly ensures the previously prescribed properties but also that 
\begin{equation}\label{eq:approx_u0_higher_properties}
	u_{0,\eps} \rightarrow u_0 \;\;\;\;  \text{ in } L^\uparam(\Omega) \text{ as } \eps \searrow 0
\end{equation}
as well. Similarly if $\tau = 1$, we let \begin{equation}\label{eq:v0_w0_definition}
	v_{0,\eps} \defs e^{\eps(\laplace - \beta)}v_0 \defs e^{-\eps\beta}e^{\eps\laplace}v_0 \in C^\infty(\overline{\Omega})
	\stext{ and } w_{0,\eps} \defs e^{\eps(\laplace - \delta)}w_0 \defs e^{-\eps\delta}e^{\eps\laplace}w_0 \in C^\infty(\overline{\Omega})
\end{equation}
for all $\eps \in (0,1)$.
These approximate functions are nonnegative due to the maximum principle and have the following convergence property due to the continuity of the semigroup at $t = 0$:
\begin{equation}\label{eq:approx_vw0_properties}
	v_{0,\eps} \rightarrow v_0 \stext{ and } w_{0,\eps} \rightarrow w_0 \;\;\;\; \text{ in } W^{1,r}(\Omega) \text{ as } \eps\searrow 0.
\end{equation}
As a convenient by-product of this construction, we also gain that 
\begin{equation}\label{eq:approx_data_mass_conservation}
	\int_\Omega v_{0,\eps} \leq \int_\Omega v_0 \stext{ and } \int_\Omega w_{0,\eps} \leq \int_\Omega w_0
\end{equation}
again due to the mass conservation property of the Neumann heat semigroup.
\\[0.5em]
As it will not only be a useful tool in arguing that our approximate solutions are in fact global, but will also play a key part in many of our later derivations of a priori estimates, we will now introduce the following transformation for classical solutions to (\ref{problem}): For any such solution $(u,v,w)$, we let
\begin{equation}\label{eq:z}
	z(x,t) \defs \xi w(x,t) - \chi v(x,t) \stext{ and } \zeta \defs \xi\gamma  - \chi\alpha \in \R \stext{ as well as } \sigma \defs \chi(\beta-\delta) \in \R
\end{equation}
for all $x \in \overline{\Omega}$, $t \in [0,\infty)$. Then $(u,z,v)$  solves the related system 
\begin{equation}\label{problem_simplified}
	\left\{
	\begin{aligned}
		u_t &= \laplace u + \div (u\grad z) && \text{ on } \Omega \times (0,\infty), \\
		\tau z_t &= \laplace z - \delta z + \zeta u + \sigma v && \text{ on } \Omega \times (0,\infty), \\
		\tau v_t &= \laplace v + \alpha u - \beta v  && \text{ on } \Omega \times (0,\infty), \\
		0 &= \grad u \cdot \nu = \grad z \cdot \nu = \grad v \cdot \nu  && \text{ on } \partial \Omega \times (0,\infty).
	\end{aligned}
	\right.
\end{equation} 
In part using this, we can now formulate the necessary existence result for our approximate solutions.
\begin{lemma}\label{lemma:approx_exist}
	There exists $\Capprox > 0$ such that, if we assume (\ref{scenario_1}), (\ref{scenario_2}) with $C_{\mathrm{S2}} \leq \Capprox$ or (\ref{scenario_3}), then the following holds:
	\\[0.5em]
	For each $\eps \in (0,1)$, there exist a positive function $\ue  \in C^0(\overline{\Omega}\times[0,\infty))\cap C^{2,1}(\overline{\Omega}\times(0,\infty))$ and nonnegative functions $\ve, \we \in C^0(\overline{\Omega}\times[0,\infty))\cap C^{2,1}(\overline{\Omega}\times(0,\infty))$ such that  $(\ue, \ve, \we)$ is a classical solution to (\ref{problem}) with initial data $u_{0,\eps}$ as well as initial data $v_{0,\eps}$ and $w_{0,\eps}$ if $\tau = 1$. Further,
	\begin{equation}\label{eq:mass_conservation}
		\int_\Omega \ue(\cdot, t)	= \int_\Omega u_{0,\eps} = u_0(\overline{\Omega}) = m 
	\end{equation}
	as well as
	\begin{equation}\label{eq:mass_bound}
		\int_\Omega \ve(\cdot, t) \leq \max\left( \frac{\alpha m}{\beta}, \tau \int_\Omega v_0 \right) \stext{ and } \int_\Omega \we(\cdot, t) \leq \max\left( \frac{\gamma m}{\delta}, \tau \int_\Omega w_0 \right)
	\end{equation}
	for all $t \in (0,\infty)$.
\end{lemma}
\begin{proof}
	As the system described in (\ref{problem}) is the same as the one discussed in \cite{TaoCOMPETINGEFFECTSATTRACTION2013}, we can in fact use Lemma 3.1 from the aforementioned paper to ensure that, for each $\eps \in (0,1)$, nonnegative local solutions exist on some time interval $(0, \tmaxeps)$ and, if $\tmaxeps < \infty$, then $\limsup_{t\nearrow \tmax}\|\ue(\cdot, t)\|_\L{\infty} = \infty$. The mass conservation property (\ref{eq:mass_conservation}) and mass boundedness properties in (\ref{eq:mass_bound}) then immediately follow by integrating the equations in (\ref{problem}). Positivity of $\ue$ is further a direct consequence of the maximum principle.
  	\\[0.5em]
	If $\xi\gamma  - \chi\alpha > 0$, finite-time blow-up of the above solutions has already been ruled out in \cite[Theorem 1.1]{LiuGlobalBoundednessFully2015} or \cite[Theorem 2.1]{TaoCOMPETINGEFFECTSATTRACTION2013} for all of our scenarios.
	\\[0.5em]
	If $\xi\gamma  - \chi\alpha = 0$, then $\zeta = 0$ and thus the second equation in the closely related problem (\ref{problem_simplified}) does not directly depend on $\ue$ anymore, which makes the system much less challenging. In two and three dimensions, we can thus use either elliptic regularity theory or semigroup methods to first conclude that there exist bounds for $\ve$ in $L^p(\Omega)$ for all $p\in(1,n)$ using the mass conservation property (\ref{eq:mass_conservation}), which by the same line of reasoning gives us a bound for $\ze$ in $W^{1,p}(\Omega)$ for any $p\in(1,\infty)$ and $\ze$ defined as in (\ref{eq:z}) (cf.\ \Cref{lemma:gradz_bound} for a similar argument). Similar to the reasoning employed in \Cref{lemma:u_linfty_bound} another semigroup based argument then allows us to rule out finite-time blow-up in this case altogether.
	\\[0.5em]
	If $\xi\gamma  - \chi\alpha < 0$, we are necessarily in Scenario (\ref{scenario_2}). As thus $n = 2$, we can use the same elliptic regularity theory for $L^1(\Omega)$ source terms (cf.\ \cite[Lemma 23]{BrezisSemilinearSecondorderElliptic1973}) we will later also use in \Cref{lemma:baseline_elliptic} to fix a constant $K_1(p) > 0$ for each $p\in (1,\infty)$ such that $\int_\Omega \we^p \leq K_1(p)$ on $(0,\tmaxeps)$ due to the mass conservation property (\ref{eq:mass_conservation}). Further using a straightforward consequence of the Gagliardo--Nirenberg inequality (cf.\ \cite[Lemma 4.2]{HeihoffExistenceGlobalSmooth2021}), we then fix $K_2(p) > 0$ for each $p \in (1,\infty)$ such that
	\begin{equation}\label{eq:gni_exist}
		\int_\Omega \ue^{p+1} \leq K_2(p) m \int_\Omega |\grad \ue^\frac{p}{2}|^2 + K_2(p) m^{p+1}
	\end{equation}
	for all $\eps \in (0,1)$ and $t\in(0,\tmaxeps)$, where $K_2(p)$ only depends on $p$ and the domain $\Omega$. Having fixed these constants, we now let $\Capprox \defs \frac{1}{4 K_2(8)}$.
	\\[0.5em]
	We then test the first equation in (\ref{problem}) with $\ue^{p-1}$, use partial integration and then use the second and third equation in (\ref{problem}), the estimate in (\ref{scenario_2}) with $C_{\mathrm{S2}} \leq \Capprox$ as well as Young's inequality to gain
	\begin{align*}
		\frac{1}{p(p-1)}\frac{\d}{\d t}\int_\Omega \ue^p 
		&= -\frac{4}{p^2}\int_\Omega |\grad \ue^\frac{p}{2}|^2 + \frac{\chi}{p}\int_\Omega \grad \ue^p \cdot \grad \ve - \frac{\xi}{p}\int_\Omega \grad \ue^p \cdot \grad \we \\
		&= -\frac{4}{p^2}\int_\Omega |\grad \ue^\frac{p}{2}|^2 - \frac{\chi}{p}\int_\Omega \ue^p \laplace \ve + \frac{\xi}{p}\int_\Omega \ue^p \laplace \we \\
		&= -\frac{4}{p^2}\int_\Omega |\grad \ue^\frac{p}{2}|^2 - \frac{\zeta}{p}\int_\Omega \ue^{p+1} - \frac{\chi \beta}{p}\int_\Omega \ve \ue^p + \frac{\xi \delta}{p}\int_\Omega \we \ue^p \\
		&\leq -\frac{4}{p^2}\int_\Omega |\grad \ue^\frac{p}{2}|^2 + \frac{\Capprox}{m p}\int_\Omega \ue^{p+1} +\frac{\xi \delta}{p}\int_\Omega \we \ue^p \\
		&\leq -\frac{4}{p^2}\int_\Omega |\grad \ue^\frac{p}{2}|^2 + \frac{2\Capprox}{m p}\int_\Omega \ue^{p+1} + K_3(p) m^p\int_\Omega \we^{p+1} \\
		&\leq -\frac{4}{p^2}\int_\Omega |\grad \ue^\frac{p}{2}|^2 + \frac{2\Capprox}{m p}\int_\Omega \ue^{p+1} + K_1(p+1)K_3(p) m^p\numberthis \label{eq:local_exist_test}
	\end{align*}
	for all $\eps \in (0,1)$, $t\in(0,\tmaxeps)$ and $p\in(2,\infty)$ with $K_3(p) \defs (\frac{\xi \delta}{p})^{p+1}(\frac{p}{\Capprox})^p$. If we now apply (\ref{eq:gni_exist}) to the above and set $p = 8$, we gain  
	\[
		\frac{1}{56}\frac{\d}{\d t}\int_\Omega \ue^8 \leq K_4
	\]
	for all $\eps \in (0,1)$ and $t\in(0,\tmaxeps)$ with $K_4 \defs (\frac{1}{16} + K_1(9)K_3(8) )m^8$ by our choice of $\Capprox$. Thus by time integration and standard elliptic regularity theory (cf.\ \cite{FriedmanPartialDifferentialEquations1969}) applied to the second and third equation in (\ref{problem}), it follows that there exists $K_5 > 0$ such that  
	\[
		\|\ue(\cdot, t)\|_\L{8} \leq K_5, \;\;\;\;\|\ve(\cdot, t)\|_{W^{2,8}(\Omega)} \leq K_5 \stext{ as well as }\|\we(\cdot, t)\|_{W^{2,8}(\Omega)} \leq K_5
	\]
	for all $\eps \in (0,1)$ and $t\in(0,\tmaxeps)$ if $\tmaxeps$ is finite.
	With $\ze$ defined as in (\ref{eq:z}), this directly gives us
	\[
		\|\ze(\cdot, t)\|_{W^{2,8}(\Omega)} \leq K_5(|\xi| + |\chi|) \sfed K_6 
	\]
	for all $\eps \in (0,1)$ and $t\in(0,\tmaxeps)$ if $\tmaxeps$ is finite. Using the variation-of-constants representation for $\ue$ corresponding to the first equation in (\ref{problem_simplified}) and the smoothing properties of the Neumann heat semigroup (cf.\ \cite[Lemma 1.3]{WinklerAggregationVsGlobal2010}), it further follows that 
	\begin{align*}
		\|\ue(\cdot, t)\|_\L{\infty} &\leq \left\|\, e^{t\laplace} u_{0,\eps} + \int_0^t e^{(t-s)\laplace} \div (\ue(\cdot, s) \grad \ze(\cdot, s)) \d s \,\right\|_\L{\infty} \\ 
		&\leq K_6 \|u_{0,\eps}\|_\L{\infty} + K_7 \int_0^t (1 + (t-s)^{-\frac{3}{4}}) \|\ue(\cdot, t) \grad \ze(\cdot, s)\|_\L{4} \d s \\
		&\leq K_6 \|u_{0,\eps}\|_\L{\infty} +  K_5 K_6 K_7 \int_0^t (1+ s^{-\frac{3}{4}}) \d s
	\end{align*}
	for all $\eps \in (0,1)$ and $t\in(0,\tmaxeps)$ with some appropriate constant $K_7 > 0$ if $\tmaxeps$ is finite. As the remaining integral is bounded for $t\in(0,\tmaxeps)$ if $\tmaxeps$ is finite, the above inequality in fact rules out finite-time blow-up in this scenario as well and thus completes the proof.
\end{proof}
\noindent
As we will henceforth always work in at least one of the three scenarios (\ref{scenario_1}), (\ref{scenario_2}) with a constant smaller than the one introduced in the above lemma or (\ref{scenario_3}), we will also fix the solutions constructed above as $(\ue, \ve, \we)$ for all $\eps\in(0,1)$ as a matter of convenience for the remainder of the paper. We further always correspondingly define $\ze$ as in (\ref{eq:z}).

\section{A priori estimates degrading toward zero}
\label{section:apriori}
As is typical for a construction of this kind, we will now spend the remainder of this paper deriving sufficient a priori estimates to gain our desired solutions as limits of the approximate solutions fixed in the previous section. Notably as a consequence of the low initial data regularity, most of these bounds necessarily need to decay as $t\searrow 0$ if they are to be independent of $\eps$. While we will see that the degree to which this decay happens is not important to ensure that our limit functions solve (\ref{problem}) on $\Omega\times(0,\infty)$, we will in fact need more qualitative information regarding the decay to ensure that the limit functions are continuous at $t = 0$ in the sense laid out in (\ref{u_continuity}), (\ref{v_continuity}) and (\ref{w_continuity}). As such, the aim of this section is to derive exactly such uniform a priori information for $\ue$, $\ve$, $\we$ and $\ze$.
\\[0.5em]
As the parabolic-parabolic and parabolic-elliptic cases call for very different methods to establish this important baseline information, we will address them here separately.  

\subsection{The parabolic-parabolic case}

Before we approach the derivation of the titular a priori estimates decaying close to time zero for the parabolic-parabolic case, we first derive the best to be expected Sobolev bounds for $\ve$, $\we$ and $\ze$ that hold up to time zero corresponding with the initial data regularity for $v_0$ and $w_0$ as well as the mass conservation property (\ref{eq:mass_conservation}). These bounds will later serve as a useful baseline for interpolation. 

\begin{lemma}\label{lemma:baseline_parabolic}
	Assume we are in Scenario (\ref{scenario_1}). Then there exists $C > 0$ such that 
	\[
		\|\ve(\cdot, t)\|_{W^{1,r}(\Omega)} \leq C, \;\;\;\; \|\we(\cdot, t)\|_{W^{1,r}(\Omega)} \leq C, \;\;\;\; \|\ze(\cdot, t)\|_{W^{1,r}(\Omega)} \leq C	
	\]
	for all $t\in[0,\infty)$ and $\eps \in (0,1)$ with $r\in(\frac{6}{5}, 2)$ as fixed at the beginning of this section.
\end{lemma}
\begin{proof}
	Given the convergence properties in (\ref{eq:approx_vw0_properties}), we can fix $K_1 > 0$ such that 
	\[
		\|\ve(\cdot, 0)\|_{W^{1,r}(\Omega)} \leq K_1.
 	\]
	Then using the variation-of-constants representation of the second equation in (\ref{problem}) and known smoothing properties of the Neumann heat semigroup (cf.\ \cite[Lemma 1.3]{WinklerAggregationVsGlobal2010}), we can fix $K_2 > 0$ such that 
	\begin{align*}
		\|\grad \ve(\cdot, t)\|_\L{r} &\leq \left\| \, \grad e^{t(\laplace-\beta)}\ve(\cdot, 0) + \alpha\int_0^t \grad e^{(t-s)(\laplace-\beta)}\ue(\cdot, s)  \d s \, \right\|_\L{r}	\\
		&\leq K_2 \|\grad \ve(\cdot, 0)\|_\L{r} + \alpha K_2\int_0^t (1+(t-s)^{\frac{1}{r}-\frac{3}{2}})e^{-\beta(t-s)} \|\ue(\cdot, s)\|_\L{1} \d s \\
		&\leq K_1 K_2 +  \alpha m K_2 \int_0^t (1+ s^{\frac{1}{r} - \frac{3}{2}})e^{-\beta s} \d s
	\end{align*}
	for all $t\in [0,\infty)$ and $\eps \in (0,1)$. As due to $r \in (1,2)$ the remaining integral term is bounded independent of $t$, our desired bound for $\ve$ follows from the above by combining it with the mass bound from (\ref{eq:mass_bound}) and e.g.\ the Poincaré inequality. By essentially the same reasoning, we obtain a corresponding bound for $\we$. The bound for $\ze$ then follows immediately as $\ze$ is merely a linear combination of $\ve$ and $\we$.
\end{proof}\noindent
Since in Scenario (\ref{scenario_1}) we work in a two-dimensional setting, the Sobolev embedding theorem directly yields the following corollary to the above, which we will use to handle the case $\sigma \neq 0$ whenever necessary.
\begin{corollary}\label{corollary:v_l2}
	Assume we are in Scenario (\ref{scenario_1}). Then for each $p\in [1,\frac{2r}{2-r}]$, there exists $C \equiv C(p) > 0$ such that 
	\[
		\|\ve(\cdot, t)\|_\L{p} \leq C	
	\]
	for all $t\in[0,\infty)$ and $\eps \in (0,1)$.
\end{corollary}  
\noindent
To now gain the central result of this section, we will employ an energy-type argument based on the familiar  functional $\mathcal{F}(t) \defs \zeta\int_\Omega \ue \ln(\ue) + \frac{1}{2}\int_\Omega |\grad \ze|^2$ multiplied by $t^{\lambda}$ with a sufficiently strong dampening factor $\lambda$ to make it initial-data independent while ensuring that the resulting additional terms can still be absorbed by the dissipative terms. This will not only allow us to derive a bound for said energy-type functional itself but more importantly some dampened space-time integral bounds corresponding to higher order terms of $\ue$ and $\ze$, which will prove crucial to connect our later constructed solutions to their initial data.
Further note that this argument centrally relies on the restriction $\zeta = \xi\gamma- \chi \alpha \geq 0$ from (\ref{scenario_1}) to ensure that the functional is always bounded from below and in fact this is the main reason this restriction is necessary for Scenario (\ref{scenario_1}). 
\begin{lemma}\label{lemma:energy}
	Assume we are in Scenario (\ref{scenario_1}). Then there exists $\lambda \in (0,\frac{2}{3})$ such that, for each $T > 0$, there is $C \equiv C(T) > 0$ with 
	\[
		\zeta\int_\Omega \ue(\cdot,t) \ln(\ue(\cdot,t)) \leq C t^{-\lambda} \stext{ and } \int_\Omega |\grad \ze(\cdot,t)|^2 \leq C t^{-\lambda} 
	\]
	as well as 
	\[
		\zeta\int_0^t s^{\lambda} \int_\Omega \frac{|\grad \ue(x,s)|^2}{\ue(x,s)} \d x \d s \leq C 	 \stext{ and } \int_0^t s^{\lambda} \int_\Omega |\laplace \ze(x,s)|^2 \d x \d s \leq C 	
	\]
	for all $t \in (0,T)$ and $\eps \in (0,1)$.
\end{lemma}
\begin{proof}
We begin by fixing $\lambda \in (0,\frac{2}{3})$ such that 
\[
	\lambda - \frac{2}{r} > -1,
\]  
which is possible as $r > \frac{6}{5}$.
\\[0.5em]
We now test the first equation in (\ref{problem_simplified}) with $t^{\lambda}\ln(\ue)$ and use partial integration to gain
\begin{equation}\label{eq:lnu_test}
	\frac{\d}{\d t} \left[ t^{\lambda} \int_\Omega \ue \ln(\ue)\right] = -t^{\lambda}\int_\Omega \frac{|\grad \ue|^2}{\ue} - t^{\lambda}\int_\Omega \grad \ue \cdot \grad \ze + \lambda t^{\lambda - 1}\int_\Omega \ue \ln(\ue)
\end{equation}
for all $t \in (0,\infty)$ and $\eps \in (0,1)$. We now fix $p \in (1,2)$ such that 
\[ 
	\frac{1-p}{2-p} + \lambda > 0
\]
and then employ the Gagliardo--Nirenberg inequality as well as Young's inequality combined with the mass conservation property (\ref{eq:mass_conservation}) to gain $K_1 > 0$ such that 
\begin{align*}
	\int_\Omega \ue \ln(\ue) &\leq \frac{1}{e(p-1)}\int_\Omega \ue^p = \frac{1}{e(p-1)} \|\ue^\frac{1}{2}\|^{2p}_\L{2p} \\
	&\leq K_1 \|\grad \ue^\frac{1}{2}\|^{2p - 2}_\L{2} + K_1 = \frac{K_1}{2^{2p - 2}} \left(\int_\Omega \frac{|\grad \ue|^2}{\ue}\right)^{p - 1} + K_1 \\
	&\leq \frac{t}{2\lambda} \int_\Omega \frac{|\grad \ue|^2}{\ue} + K_2t^\frac{1-p}{2-p} + K_1
\end{align*}
for all $t \in (0,\infty)$ and $\eps \in (0,1)$ with $K_2 \defs 2^\frac{1-p}{2-p} \lambda^\frac{p-1}{2-p} K_1^\frac{1}{2-p}$. We then apply this to (\ref{eq:lnu_test}) to gain
\begin{equation}\label{eq:lnu_test2}
	\frac{\d}{\d t} \left[t^{\lambda} \int_\Omega \ue \ln(\ue)\right] \leq -\frac{t^{\lambda}}{2}\int_\Omega \frac{|\grad \ue|^2}{\ue} - t^{\lambda}\int_\Omega \grad \ue \cdot \grad \ze + \lambda K_2 t^{\frac{1-p}{2-p} + \lambda - 1} + \lambda K_1 t^{\lambda - 1}
\end{equation}
for all $t\in (0,\infty)$ and $\eps \in (0,1)$.
\\[0.5em]
As our next step, we test the second equation in (\ref{problem_simplified}) with $-t^{\lambda} \laplace \ze$ and use partial integration to gain 
\begin{align*}
	&\frac{\d}{\d t} \left[ \frac{t^{\lambda}}{2} \int_\Omega |\grad \ze|^2 \right] \\
	=& -t^{\lambda} \int_\Omega |\laplace \ze|^2 - \delta t^{\lambda} \int_\Omega |\grad \ze|^2 + \zeta t^{\lambda} \int_\Omega \grad \ze \cdot \grad \ue - \sigma t^{\lambda} \int_\Omega \ve\laplace \ze + \tfrac{\lambda}{2} t^{\lambda - 1} \int_\Omega |\grad \ze|^2 \numberthis\label{eq:laplacez_test}
\end{align*}
for all $t \in (0,\infty)$ and $\eps \in (0,1)$. Another application of the Gagliardo--Nirenberg inequality as well as Young's inequality combined with elliptic regularity theory and the baseline bound in $W^{1,r}(\Omega)$ established in \Cref{lemma:baseline_parabolic} further yields $K_3 > 0$ such that
\begin{align*}
	\int_\Omega |\grad \ze|^2 &= \|\grad \ze\|^2_\L{2} \leq K_3 \|\laplace \ze\|^{2-r}_\L{2} + K_3 = K_3 \left(\int_\Omega |\laplace \ze|^2 \right)^{1-\frac{r}{2}} + K_3 \\
	&\leq \frac{t}{\lambda} \int_\Omega |\laplace \ze|^2 + K_4 t^{1-\frac{2}{r}} + K_3
\end{align*}
for all $t \in (0,\infty)$ and $\eps \in (0,1)$ with $K_4 \defs K_3^{\frac{2}{r}} \lambda^{\frac{2}{r} - 1}$. If we now apply this as well as Young's inequality and \Cref{corollary:v_l2} to (\ref{eq:laplacez_test}), we gain 
\begin{align*}
	&\frac{\d}{\d t} \left[ \frac{t^{\lambda}}{2} \int_\Omega |\grad \ze|^2 \right] \\
	\leq& -\frac{t^{\lambda}}{2} \int_\Omega |\laplace \ze|^2 + \zeta t^{\lambda} \int_\Omega \grad \ze \cdot \grad \ue - \sigma t^{\lambda} \int_\Omega \ve \laplace \ze  +  \frac{\lambda K_4}{2} t^{\lambda-\frac{2}{r}} + \frac{\lambda K_3}{2}t^{\lambda - 1} \\
	%
	%
	%
	%
	%
	%
	%
	%
	\leq& -\frac{t^{\lambda}}{4} \int_\Omega |\laplace \ze|^2 + \zeta t^{\lambda} \int_\Omega \grad \ze \cdot \grad \ue + \sigma^2 K_5^2 t^{\lambda} +  \frac{\lambda K_4}{2} t^{\lambda-\frac{2}{r}} + \frac{\lambda K_3}{2}t^{\lambda - 1} 
\end{align*}
for all $t \in (0,\infty)$ and $\eps \in (0,1)$ with $K_5 > 0$ as provided by \Cref{corollary:v_l2} as $\frac{2r}{2-r} \geq 2$. Combining this with an appropriately scaled (\ref{eq:lnu_test2}) yields 
\begin{align*}
	&\frac{\d}{\d t} \left[ \zeta t^{\lambda} \int_\Omega \ue \ln(\ue) + \frac{t^{\lambda}}{2}\int_\Omega |\grad \ze|^2 \right] + \zeta\frac{t^{\lambda}}{2} \int_\Omega \frac{|\grad \ue|^2}{\ue} + \frac{ t^{\lambda}}{4} \int_\Omega |\laplace \ze|^2 \\
	& \;\;\;\;\;\; \leq K_6 \left[ t^{\frac{1-p}{2-p} + \lambda - 1} + t^{\lambda - \frac{2}{r}}  + t^{\lambda - 1} + t^{\lambda} \right]
\end{align*}
for all $t \in (0,\infty)$ and $\eps \in (0,1)$ with $K_6 \defs \max(\zeta\lambda K_2, \frac{\lambda K_4}{2}, \zeta \lambda K_1 + \frac{\lambda K_3}{2}, \sigma^2 K_5^2)$. Given that our choices of $p$ and $\lambda$ ensure that the exponents on the right side of the above inequality are all larger than negative one and $\zeta \geq 0$, our desired result follows immediately by time integration.
\end{proof}\noindent
The above result now leads directly into a straightforward but important corollary.
\begin{corollary}\label{corollary:energy}
Assume we are in Scenario (\ref{scenario_1}). Then for each $T > 0$, there exists $C \equiv C(T) > 0$ such that 
\[
	\int_0^t s^{2\lambda} \int_\Omega |\grad \ze(x,s)|^4 \d x \d s \leq C
\]
for all $t \in (0,T)$ and $\eps \in (0,1)$ with $\lambda$ as in \Cref{lemma:energy}.
\end{corollary}
\begin{proof}
	Using the Gagliardo--Nirenberg inequality, we see that 
	\begin{align*}
		t^{2\lambda}\int_\Omega |\grad \ze|^4 =& t^{2\lambda} \|\grad \ze\|^4_\L{4} \leq Kt^{2\lambda}\|\laplace \ze\|^2_\L{2} \|\grad \ze\|^2_\L{2} + 
		K t^{2\lambda}\|\grad \ze\|^4_\L{2} \\
		=& K \left(t^{\lambda}\int_\Omega |\laplace \ze|^2 \right) \left( t^{\lambda}\int_\Omega |\grad \ze|^2 \right) + K \left( t^{\lambda} \int_\Omega |\grad \ze|^2 \right)^2
	\end{align*}
	for all $\eps \in (0,1)$ and $t\in(0,\infty)$ with some appropriate $K > 0$. Given the bounds already established in \Cref{lemma:energy}, this immediately gives us our desired result.
\end{proof}
\noindent
While we have now already established all the estimates 
necessary to ensure continuity of our later constructed solutions at time $t=0$, we will now derive a further $L^2(\Omega)$ estimate for $\ue$ away from $t=0$ to facilitate the uniform handling of both the elliptic and parabolic cases in the next section, which is devoted to the derivation of regularity away from $t=0$. This is achieved by adapting a standard testing argument for parabolic-parabolic repulsion systems in a similar manner to the above energy-based arguments.
\begin{lemma}\label{lemma:u2_bound}
Assume we are in Scenario (\ref{scenario_1}). Then for each $t_0, t_1 \in \R$ with $t_1 > t_0 > 0$, there exists $C \equiv C(t_1, t_0) > 0$ such that
\[
	\| \ue(\cdot, t) \|_\L{2} \leq C
\]
for all $t\in(t_0,t_1)$ and $\eps \in (0,1)$.
\end{lemma}
\begin{proof}
	We begin by fixing $t_1 > t_0 > 0$ and then use \Cref{corollary:energy} to further fix $K_1 \equiv K_1(t_0, t_1) > 0$ such that 
	\[
		\int_{\frac{t_0}{2}}^{t_1} \int_\Omega |\grad \ze|^4 \leq K_1
	\]
	for all $\eps \in (0,1)$.
	\\[0.5em]
	We now test the first equation in (\ref{problem_simplified}) with $(t - \frac{t_0}{2})^2\ue$ to gain 
	\begin{align*}
		\frac{\d}{\d t} \left[ \frac{(t - \frac{t_0}{2})^2}{2}\int_\Omega \ue^2 \right] 	&= -(t - \tfrac{t_0}{2})^2 \int_\Omega |\grad \ue|^2 - (t - \tfrac{t_0}{2})^2\int_\Omega \ue \grad \ue \cdot \grad \ze + (t-\tfrac{t_0}{2})\int_\Omega \ue^2 \\
		&\leq -\frac{(t - \tfrac{t_0}{2})^2}{2} \int_\Omega |\grad \ue|^2 + \frac{(t - \tfrac{t_0}{2})^2}{2}\int_\Omega \ue^2 |\grad \ze|^2 + (t-\tfrac{t_0}{2})\int_\Omega \ue^2 \\
		&\leq  -\frac{(t - \tfrac{t_0}{2})^2}{2} \int_\Omega |\grad \ue|^2 + \frac{(t - \tfrac{t_0}{2})^2}{2}\left(\int_\Omega \ue^4\right)^\frac{1}{2} \left( \int_\Omega |\grad \ze|^4 \right)^\frac{1}{2} + (t-\tfrac{t_0}{2})\int_\Omega \ue^2 \numberthis \label{eq:u2_test}
	\end{align*}
	for all $t\in(\frac{t_0}{2}, t_1)$ and $\eps \in (0,1)$ by applying the Young and Hölder inequalities.
	\\[0.5em]
	Using the Gagliardo--Nirenberg inequality combined with the mass conservation property (\ref{eq:mass_conservation}), we now fix $K_2 > 1$ such that  
	\begin{align*}
		\int_\Omega \ue^2 = \|\ue\|^2_\L{2} &\leq K_2 \|\grad \ue\|_\L{2} + K_2 = K_2 \sqrt{\int_\Omega |\grad \ue|^2} + K_2 \\
		&\leq \frac{(t-\frac{t_0}{2})}{4} \int_\Omega |\grad \ue|^2 + (t-\tfrac{t_0}{2})^{-1} K_2^2 + K_2 
	\end{align*}
	and 
	\begin{align*}
		\int_\Omega \ue^4 = \|\ue\|^4_\L{4} \leq K_2^2 \|\grad \ue\|^2_\L{2}\|\ue\|^2_\L{2} + K_2^2 = K_2^2 \left( \int_\Omega |\grad \ue|^2 \right) \left( \int_\Omega \ue^2 \right) + K_2^2  
	\end{align*}
	for all $t \in (\tfrac{t_0}{2},t_1)$ and $\eps \in (0,1)$. Applying these inequalities to (\ref{eq:u2_test}) then yields
	\begin{align*}
		\frac{\d}{\d t} \left[ \frac{(t - \frac{t_0}{2})^2}{2}\int_\Omega \ue^2 \right] &\leq -\frac{(t - \tfrac{t_0}{2})^2}{4} \int_\Omega |\grad \ue|^2  + \frac{K_2(t - \tfrac{t_0}{2})^2}{2}   \left(\int_\Omega |\grad \ue|^2 \right)^\frac{1}{2} \left(\int_\Omega \ue^2\right)^\frac{1}{2}\left( \int_\Omega |\grad \ze|^4 \right)^\frac{1}{2} 
		\\
		&+ \frac{K_2(t - \tfrac{t_0}{2})^2}{2} \left( \int_\Omega |\grad \ze|^4 \right)^\frac{1}{2} + K_3 \\
		&\leq K_4 (t - \tfrac{t_0}{2})^2  \left( \int_\Omega |\grad \ze|^4 \right) \left(\int_\Omega \ue^2 \right) + K_4\left( \int_\Omega |\grad \ze|^4 \right) + K_4 \\
		&=  K_4 \left( \int_\Omega |\grad \ze|^4 \right) \left( 1 +  (t - \tfrac{t_0}{2})^2\int_\Omega \ue^2 \right) + K_4
	\end{align*}
	for all $t \in (\tfrac{t_0}{2},t_1)$ and $\eps \in (0,1)$ with $K_3 \equiv K_3(t_0, t_1) \defs K_2^2 + (t_1 - \frac{t_0}{2}) K_2$ and $K_4 \equiv K_4(t_0, t_1) \defs  \max(\frac{K_2^2}{4}, K_3 + (t_1 - \tfrac{t_0}{2})^4)$. Setting $y_\eps(t) \defs 1 + (t - \frac{t_0}{2})^2\int_\Omega \ue^2(\cdot, t)$ for all $t\in[\tfrac{t_0}{2},t_1)$, the above implies that the function $y_\eps$ satisfies
	\[
		y_\eps'(t) \leq 2K_4 \left( 1 + \int_\Omega |\grad \ze(\cdot ,t)|^4 \right)y_\eps(t)
	\]
	for all $t \in (\tfrac{t_0}{2},t_1)$ and $\eps \in (0,1)$ as well as $y_\eps(\tfrac{t_0}{2}) = 1$ for all $\eps \in (0,1)$. By a standard comparison with the explicit solution to the differential equality corresponding to the above inequality, this then yields 
	\begin{align*}
		(t - \tfrac{t_0}{2})^2 \int_\Omega \ue^2(\cdot, t) &\leq 1 + (t - \tfrac{t_0}{2})^2 \int_\Omega \ue^2(\cdot, t) = y_\eps(t)\\
		&\leq \exp\left( 2K_4(t-\tfrac{t_0}{2}) + 2K_4\int_\frac{t_0}{2}^t \int_\Omega |\grad \ze(\cdot, s)|^4 \d s  \right) \\
		&\leq \exp\left( 2K_4(t_1-\tfrac{t_0}{2})  + 2K_1K_4\right) \sfed K_5
	\end{align*}
	for all $t \in (\tfrac{t_0}{2},t_1)$ and $\eps \in (0,1)$. Thus
	\[
		\int_\Omega \ue^2(\cdot, t) \leq \frac{4K_5}{t_0^2}
	\]
	for all $t \in (t_0,t_1)$ and $\eps \in (0,1)$, which is sufficient to complete the proof.
\end{proof}

\subsection{The parabolic-elliptic case}

We now turn our attention to the parabolic-elliptic case.
Here, our first step again entails the derivation of bounds for $\ve$, $\we$ and $\ze$ up to $t=0$ as a baseline for future interpolation. This time our argument is naturally based on elliptic regularity theory as opposed to its parabolic counterparts. 
\begin{lemma}\label{lemma:baseline_elliptic}
	Assume we are in Scenario (\ref{scenario_2}) with $C_{\mathrm{S2}} \leq \Capprox$ and $\Capprox$ as in \Cref{lemma:approx_exist} or in Scenario (\ref{scenario_3}). Then for each $p \in (1, \frac{n}{n-1})$, there exists $C_1 \equiv C_1(p) > 0$ such that 
	\[
		\|\ve(\cdot, t)\|_{W^{1,p}(\Omega)}  \leq C_1, \;\;\;\; \|\we(\cdot,t)\|_{W^{1,p}(\Omega)}  \leq C_1, \;\;\;\; \|\ze(\cdot, t)\|_{W^{1,p}(\Omega)} \leq C_1	
	\]
	and, for each $q \in (1, \frac{n}{n-2})$, there exists $C_2 \equiv C_2(q) > 0$ such that 
	\[
		\|\ve(\cdot, t)\|_\L{q} \leq C_2, \;\;\;\; \|\we(\cdot, t)\|_\L{q} \leq C_2, \;\;\;\; \|\ze(\cdot, t)\|_\L{q} \leq C_2	
	\]
	for all $t\in[0,\infty)$ and $\eps \in (0,1)$.
\end{lemma}
\begin{proof}
	This is a direct consequence of the mass conservation and boundedness properties in (\ref{eq:mass_conservation}) and (\ref{eq:mass_bound}) combined with the elliptic regularity theory for $L^1(\Omega)$ source terms from \cite[Lemma 23]{BrezisSemilinearSecondorderElliptic1973} applied to the second and third equation in (\ref{problem}) and the embedding properties of Sobolev spaces as well as the fact that $\ze$ is merely a linear combination of $\ve$ and $\we$.
\end{proof}
\noindent
As the system (\ref{problem}) in its elliptic variant will present us with very similar structural challenges and advantages as the system discussed in \cite{HeihoffExistenceGlobalSmooth2021}, we will largely follow the approach from the aforementioned reference to derive the key result of this section. To not unnecessarily reiterate already established reasoning, we will give the following argument in a complete but fairly swift fashion and refer the reader to \cite{HeihoffExistenceGlobalSmooth2021} for some of the more involved parameter calculations.
\\[0.5em]
This argument is also where most of the assumptions in Scenario (\ref{scenario_2}) and (\ref{scenario_3}) become important. Notably, it is here where the final value of $C_{\mathrm{S2}}$ is being fixed as a refinement of our choice of $\Capprox$ in \Cref{lemma:approx_exist}.
\begin{lemma}\label{lemma:elliptic_bound}
	There exists a constant $C_{\mathrm{S2}} > 0$ smaller than the constant $\Capprox$ introduced in \Cref{lemma:approx_exist} such that, if we are in Scenario (\ref{scenario_2}) with said constant $C_{\mathrm{S2}}$ or in Scenario (\ref{scenario_3}), the following holds:
	\\[0.5em]
	For each $T > 0$, there exists $C \equiv C(T) > 0$ such that 
	\[
		\int_\Omega \ue^p(\cdot,t) \leq Ct^{-\frac{n(p-1)}{2}}	
	\]
	for $p\in\{n, \frac{5}{2}, \frac{3\uparam}{4\uparam-3}\}$ and, if $u_0 \in L^\uparam(\Omega)$,
	\[
		\|\ue(\cdot, t)\|_\L{\uparam} \leq C	
	\]
	for all $t \in (0,T)$ and $\eps \in (0,1)$.
\end{lemma}
\begin{proof}
	To make sure that our choice of $C_{\mathrm{S2}}$ is sound, we will now fix $C_{\mathrm{S2}}$ before starting our argument proper. To this end, we fix $K_1 > 0$ such that 
	\begin{equation}\label{eq:ue_pp1_gni}
		\int_\Omega \ue^{p+\frac{2}{n}} \leq K_1 m^\frac{2}{n} \int_\Omega |\grad \ue^\frac{p}{2}|^2 + K_1 m^{p+\frac{2}{n}}
	\end{equation}
	for all $t\in(0,\infty)$, $\eps \in (0,1)$ and $p\in\{n, \frac{5}{2}, \uparam, \frac{3\uparam}{4\uparam-3}\}$ using the consequence of the Gagliardo--Nirenberg inequality laid out in \cite[Lemma 4.2]{HeihoffExistenceGlobalSmooth2021}. We then choose $C_{\mathrm{S2}} > 0$ such that $C_{\mathrm{S2}} \leq \frac{2}{p K_1}$ for all $p\in\{n, \frac{5}{2}, \uparam, \frac{3\uparam}{4\uparam-3}\}$ and such that it is smaller than the constant $\Capprox$ introduced in \Cref{lemma:approx_exist}.
	\\[0.5em]
	Similar to some of the arguments presented in \Cref{lemma:approx_exist} and due to the fact that both of our considered scenarios are of parabolic-elliptic type, we begin by testing the first equation in (\ref{problem}) with $\ue^{p-1}$ and then use the second and third equation in (\ref{problem}) to gain that 
	\begin{align*}
		\frac{1}{p(p-1)} \frac{\d}{\d t} \int_\Omega \ue^p &= -\frac{4}{p^2} \int_\Omega |\grad \ue^\frac{p}{2}|^2 - \frac{\chi}{p} \int_\Omega \ue^{p}\laplace \ve + \frac{\xi}{p} \int_\Omega \ue^{p}\laplace \we \\
		&= -\frac{4}{p^2} \int_\Omega |\grad \ue^\frac{p}{2}|^2 - \frac{\zeta}{p} \int_\Omega \ue^{p+1} - \frac{\chi \beta}{p} \int_\Omega \ve \ue^p + \frac{\xi \delta}{p} \int_\Omega \we \ue^p \\
		&\leq -\frac{4}{p^2} \int_\Omega |\grad \ue^\frac{p}{2}|^2 - \frac{\zeta}{p} \int_\Omega \ue^{p+1} + \frac{\xi \delta}{p} \int_\Omega \we \ue^p
		\label{eq:ue_elliptic_test_1} \numberthis
	\end{align*}
	for all $t \in (0,\infty)$, $\eps \in (0,1)$ and $p \in (1,\infty)$. 
	If we are in Scenario (\ref{scenario_3}), $\zeta \geq 0$ and thus
	\[
		- \frac{\zeta}{p} \int_\Omega \ue^{p+1} \leq 0
	\]
	for all $t \in (0,\infty)$, $\eps \in (0,1)$ and $p \in (1,\infty)$ . If we are in Scenario (\ref{scenario_2}), we can apply (\ref{eq:ue_pp1_gni}) to gain that 
	\[
		- \frac{\zeta}{p} \int_\Omega \ue^{p+1} \leq \frac{C_{\mathrm{S2}}}{m p} \int_\Omega \ue^{p+1} \leq \frac{2}{p^2} \int_\Omega |\grad \ue^\frac{p}{2}|^2 + \frac{2 m^p}{p^2}
	\]
	for all $t \in (0,\infty)$, $\eps \in (0,1)$ and $p\in\{n, \frac{5}{2}, \uparam, \frac{3\uparam}{4\uparam-3}\}$ by our previous choice of $C_{\mathrm{S2}}$ and the fact that in this scenario $n = 2$. As such in both scenarios, we gain
	\begin{equation}\label{eq:ue_elliptic_test_2}
		\frac{1}{p(p-1)} \frac{\d}{\d t} \int_\Omega \ue^p \leq -\frac{2}{p^2} \int_\Omega |\grad \ue^\frac{p}{2}|^2 + \frac{\xi \delta}{p} \int_\Omega \we \ue^p + \frac{2m^p}{p^2}
	\end{equation}
	for all $t\in (0,\infty)$, $\eps \in (0,1)$ and $p\in\{n, \frac{5}{2}, \uparam, \frac{3\uparam}{4\uparam-3}\}$ from (\ref{eq:ue_elliptic_test_1}).
	\\[0.5em]
	By employing Young's inequality again combined with (\ref{eq:ue_pp1_gni}), we further see that 
	\[
		\frac{\xi \delta}{p} \int_\Omega \we \ue^p \leq K_2 \int_\Omega \we^{1+\frac{np}{2}} + \frac{1}{p^2 m^\frac{2}{n} K_1} \int_\Omega \ue^{p + \frac{2}{n}} \leq K_2 \int_\Omega \we^{1+\frac{np}{2}}  + \frac{1}{p^2} \int_\Omega |\grad \ue^\frac{p}{2}|^2 + \frac{m^p}{p^2}
	\]
	with an appropriate $K_2 > 0$, which applied to (\ref{eq:ue_elliptic_test_2}) yields 
	\[
		\frac{1}{p(p-1)} \frac{\d}{\d t} \int_\Omega \ue^p \leq -\frac{1}{p^2} \int_\Omega |\grad \ue^\frac{p}{2}|^2 + K_2 \int_\Omega \we^{1+\frac{np}{2}} + \frac{3m^p}{p^2}
	\]
	for all $t\in (0,\infty)$, $\eps \in (0,1)$ and $p\in\{n, \frac{5}{2}, \uparam, \frac{3\uparam}{4\uparam-3}\}$. By applying a straightforward consequence of the Gagliardo--Nirenberg inequality found in \cite[Lemma 4.2]{HeihoffExistenceGlobalSmooth2021} to the above, we can then find $K_3 > 0$ such that 
	\begin{equation}\label{eq:ue_elliptic_test_3}
		\frac{1}{p(p-1)} \frac{\d}{\d t} \int_\Omega \ue^p \leq -\frac{1}{K_3} \left(\int_\Omega \ue^p \right)^{1+\frac{2}{n(p-1)}} + K_2 \int_\Omega \we^{1+\frac{np}{2}} + K_3
	\end{equation}
	for all $t\in (0,\infty)$, $\eps \in (0,1)$ and $p\in\{n, \frac{5}{2}, \uparam, \frac{3\uparam}{4\uparam-3}\}$. Yet another application of the Gagliardo--Nirenberg inequality in a fashion very similar to \cite[Lemma 4.3]{HeihoffExistenceGlobalSmooth2021} then further yields $K_4 > 0$ such that 
	\[
		\|\we\|_\L{1+\frac{np}{2}} \leq K_4 \|\we\|^\theta_{W^{2, p}}\|\we\|^{1-\theta}_\L{q}
	\]
	with $q \equiv q(p) \defs (n(n-1) + 2\frac{n}{p})(2 + \frac{n}{p})^{-1} \in (1, \frac{n}{n-2})$ and $\theta \equiv \theta(p) \defs \frac{p}{2} (1+\frac{np}{2})^{-1} \in (0,1)$ for all $t \in (0,\infty)$, $\eps \in (0,1)$ and $p\in\{n, \frac{5}{2}, \uparam, \frac{3\uparam}{4\uparam-3}\}$. Combining this with elliptic regularity theory and the baseline established in \Cref{lemma:baseline_elliptic} further yields $K_5, K_6 > 0$ such that 
	\[
		\int_\Omega \we^{1+\frac{np}{2}} \leq K_5 \|\ue\|^\frac{p}{2}_\L{p} = K_5 \left( \int_\Omega \ue^p \right)^\frac{1}{2} \leq \frac{1}{2 K_2 K_3} \left(\int_\Omega \ue^p \right)^{1+\frac{2}{n(p-1)}} 
		+ K_6
	\]
	for all $t \in (0,\infty)$, $\eps\in(0,1)$ and $p\in\{n, \frac{5}{2}, \uparam, \frac{3\uparam}{4\uparam-3}\}$. Applying this to (\ref{eq:ue_elliptic_test_3}) gives us that 
	\[
		\frac{1}{p(p-1)} \frac{\d}{\d t} \int_\Omega \ue^p \leq -\frac{1}{2 K_3} \left(\int_\Omega \ue^p \right)^{1+\frac{2}{n(p-1)}} + K_3 + K_2 K_6
	\]
	for all $t \in (0,\infty)$, $\eps\in(0,1)$ and $p\in\{n, \frac{5}{2}, \uparam, \frac{3\uparam}{4\uparam-3}\}$. From this, our desired results follow directly from either \cite[Lemma 3.1]{HeihoffExistenceGlobalSmooth2021}, which provides us with an initial data independent bound decaying towards zero for ODEs with superlinear decay, or by a straightforward time integration if $u_0 \in L^\uparam(\Omega)$ as this implies that the family $(u_{0,\eps})_{\eps\in(0,1)}$ is uniformly bounded in $L^\uparam(\Omega)$ due to (\ref{eq:approx_u0_higher_properties}). 
\end{proof}

\section{Estimates away from $t=0$ and the construction of our solution candidates}

Our aim for this section is to use the uniform bound for $\ue$ in $L^n(\Omega)$ and the uniform bound for $\grad \ze$ in $L^1(\Omega)$  away from $t=0$, which we established in the previous section for both parabolic-parabolic as well as parabolic-elliptic scenarios, as basis for a bootstrap argument to derive sufficient Hölder bounds for $\ue$, $\ve$ and $\we$ and thus construct a candidate for our desired solution via limit process. Notably, we will strive to prove sufficiently strong bounds such that said limit solutions immediately inherit the classical solution properties from the approximate solutions. To do this, we will generally focus on the more challenging parabolic cases and only give references for how similar results can be achieved in the elliptic cases where applicable.
\\[0.5em] 
We begin this process by first establishing bounds for $\ve$, $\we$ and $\ze$ in all $W^{1,p}(\Omega)$ with finite $p$ by either semigroup methods or elliptic regularity theory. Following this, we then use said bounds for a semigroup based argument to derive a uniform bound for $\ue$ in $L^\infty(\Omega)$.

\begin{lemma}\label{lemma:gradz_bound}
Assume we are in Scenario (\ref{scenario_1}), Scenario (\ref{scenario_2}) with $C_{\mathrm{S2}}$ as in \Cref{lemma:elliptic_bound} or in Scenario (\ref{scenario_3}). Then for each $p \in \left(1, \infty\right)$ and $t_0, t_1 \in \R$ with $t_1 > t_0 > 0$, there exists $C\equiv C(p, t_0, t_1) > 0$ such that
\[
	\|\ve(\cdot,t)\|_{W^{1,p}(\Omega)} \leq C, \;\;\;\; \|\we(\cdot,t)\|_{W^{1,p}(\Omega)} \leq C \stext{ and }\|\ze(\cdot,t)\|_{W^{1,p}(\Omega)} \leq C
\]
for all $t \in (t_0, t_1)$.
\end{lemma}
\begin{proof}
Fix $t_1 > t_0 > 0$ and $p\in(1,\infty)$.
\\[0.5em]
Then according to either \Cref{lemma:baseline_parabolic} and \Cref{lemma:u2_bound} or \Cref{lemma:baseline_elliptic} and \Cref{lemma:elliptic_bound} depending on the scenario, there exists $K_1 \equiv K_1(t_0, t_1) > 0$ such that 
\[
	\|\ue(\cdot, t)\|_\L{n} \leq K_1 \stext{ and } \|\grad \ve(\cdot, t)\|_\L{1} \leq K_1	
\]
for all $t \in [\frac{t_0}{2}, t_1)$. 
\\[0.5em]
If $\tau = 1$, we can use the variation-of-constants representation of the second equation in (\ref{problem}) and well-known smoothing properties of the Neumann heat semigroup (cf.\ \cite[Lemma 1.3]{WinklerAggregationVsGlobal2010}) to gain a constant $K_2 > 0$ such that 
\begin{align*} 
	\|\grad \ve(\cdot,t)\|_\L{p} &= \left\|\, \grad e^{\left(t-\frac{t_0}{2}\right)(\laplace - \beta)}\ve(\cdot, \tfrac{t_0}{2}) + \alpha\int_{t_0/2}^t \grad e^{(t-s)(\laplace - \beta)}\ue(\cdot, s) \d s \,\right\|_\L{p} \\
	&\leq K_1 K_2 \left(1 + \left(t - \tfrac{t_0}{2}\right)^{\frac{n}{2p}-\frac{n}{2}}\right) + \alpha K_1 K_2\int_{t_0/2}^t \left( 1+ (t-s)^{\frac{n}{2p}- 1}\right)\d s \\
	&\leq K_1K_2 \left(1 + \left(\tfrac{t_0}{2}\right)^{\frac{n}{2p}-\frac{n}{2}}\right) + \alpha K_1 K_2 (t_1 - \tfrac{t_0}{2}) + \alpha K_1K_2\tfrac{2p}{n}(t_1 - \tfrac{t_0}{2})^\frac{n}{2p} 
\end{align*}
for all $t\in(t_0, t_1)$ and $\eps \in (0,1)$, which gives us the first of our desired bounds when combined with the mass boundedness property (\ref{eq:mass_bound}) and the Poincaré inequality. Essentially the same argument can be applied to $\we$ to gain the second bound. As $\ze$ is merely a linear combination of $\ve$ and $\we$, the bound for $\ze$ then follows immediately and completes the proof.
\\[0.5em]
If $\tau = 0$, the same bounds follow directly from elliptic regularity theory (cf.\ \cite{FriedmanPartialDifferentialEquations1969}) combined with the Sobolev embedding theorem. 
\end{proof}

\begin{lemma}\label{lemma:u_linfty_bound}
	Assume we are in Scenario (\ref{scenario_1}), Scenario (\ref{scenario_2}) with $C_{\mathrm{S2}}$ as in \Cref{lemma:elliptic_bound} or in Scenario (\ref{scenario_3}).
	Then for each $t_0, t_1 \in \R$ with $t_1 > t_0 > 0$, there exists $C \equiv C(t_0, t_1) > 0$ such that 
	\[
		\|\ue(\cdot, t)\|_\L{\infty}\leq C
	\]
	for all $t \in (t_0, t_1)$ and $\eps \in (0,1)$.
\end{lemma}
\begin{proof}
	Fix $t_1 > t_0 > 0$.
	\\[0.5em]
	Then according to \Cref{lemma:u2_bound} or \Cref{lemma:elliptic_bound} as well as \Cref{lemma:gradz_bound}, there exists $K_1 \equiv K_1(t_0, t_1) > 0$ such that 
	\[
		\|\ue(\cdot, t)\|_\L{n} \leq K_1 \stext{ and } \|\grad \ze(\cdot, t)\|_\L{4n} \leq K_1	
	\]
	for all $t \in [\frac{t_0}{2}, t_1)$. We then further define 
	\[
		M_\eps \defs \sup_{t\in(t_0/2,t_1)} \|(t-\tfrac{t_0}{2})^\frac{1}{2}  \ue(\cdot,t)\|_\L{\infty} < \infty.
	\] 
	Using the variation-of-constants representation for the first equation in (\ref{problem_simplified}) combined with the smoothing properties of the Neumann heat semigroup (cf.\ \cite[Lemma 1.3]{WinklerAggregationVsGlobal2010}) yields a constant $K_2 > 0$ such that 
	\begin{align*}
		&\hphantom{= }\;\left\|(t-\tfrac{t_0}{2})^\frac{1}{2}\ue(\cdot,t)\right\|_\L{\infty} \\
		&= \left\| \,(t-\tfrac{t_0}{2})^\frac{1}{2}e^{\left(t-\tfrac{t_0}{2}\right)\laplace}\ue(\cdot, \tfrac{t_0}{2}) + (t-\tfrac{t_0}{2})^\frac{1}{2}\int_{t_0/2}^t e^{(t-s)\laplace} \div( \ue(\cdot, s) \grad \ze(\cdot, s)) \d s \,\right\|_\L{\infty} \\
		&\leq K_2  \|\ue(\cdot, \tfrac{t_0}{2})\|_\L{n} + K_2(t-\tfrac{t_0}{2})^\frac{1}{2}\int_{t_0/2}^t  \left(1 + (t-s)^{-\frac{3}{4}}\right)\|\ue(\cdot, s) \grad \ze(\cdot, s)\|_\L{2n} \d s \\
		&\leq K_1 K_2 + K_1 K_2 (t - \tfrac{t_0}{2})^\frac{1}{2} \int_{t_0/2}^t \left(1 +(t-s)^{-\frac{3}{4}}\right) \|\ue(\cdot, s)\|^\frac{1}{4}_\L{n} \|\ue(\cdot, s)\|_\L{\infty}^\frac{3}{4} \d s \\
		&\leq K_1 K_2 + K_1^\frac{5}{4} K_2 (t - \tfrac{t_0}{2})^\frac{1}{2} M_\eps^\frac{3}{4} \int_{t_0/2}^t \left( 1 +  (t-s)^{-\frac{3}{4}}\right) (s-\tfrac{t_0}{2})^{-\frac{3}{8}} \d s
	\end{align*}
	for all $t \in (\tfrac{t_0}{2}, t_1)$ and $\eps \in (0,1)$. As further 
	\begin{align*}
		&\int_{t_0/2}^t \left( 1 + (t-s)^{-\frac{3}{4}}\right) (s-\tfrac{t_0}{2})^{-\frac{3}{8}} \d s \\
		&= \int_{t_0/2}^{t_\text{mid}} \left( 1 + (t-s)^{-\frac{3}{4}}\right) (s-\tfrac{t_0}{2})^{-\frac{3}{8}} \d s + \int_{t_\text{mid}}^t \left( 1 + (t-s)^{-\frac{3}{4}}\right) (s-\tfrac{t_0}{2})^{-\frac{3}{8}} \d s \\
		&\leq  \left( 1 + (\tfrac{t}{2}-\tfrac{t_0}{4})^{-\frac{3}{4}}\right) \int_{t_0/2}^{t_{\text{mid}}}  (s-\tfrac{t_0}{2})^{-\frac{3}{8}} \d s+ (\tfrac{t}{2}-\tfrac{t_0}{4})^{-\frac{3}{8}}\int_{t_{\text{mid}}}^t \left( 1 + (t-s)^{-\frac{3}{4}}\right) \d s\\
		&= \tfrac{13}{5} (\tfrac{t}{2}-\tfrac{t_0}{4})^\frac{5}{8} +  \tfrac{28}{5} (\tfrac{t}{2}-\tfrac{t_0}{4})^{-\frac{1}{8}} 
	\end{align*}
	with $t_{\mathrm{mid}} \defs \frac{1}{2}\left(t + \frac{t_0}{2}\right)$ and thus 
	\[
		(t - \tfrac{t_0}{2})^\frac{1}{2}\int_{t_0/2}^t \left( 1 + (t-s)^{-\frac{3}{4}}\right) (s-\tfrac{t_0}{2})^{-\frac{3}{8}} \d s \leq \sqrt{2} \left( \tfrac{13}{5} (\tfrac{t_1}{2}-\tfrac{t_0}{4})^\frac{9}{8} +  \tfrac{28}{5} (\tfrac{t_1}{2}-\tfrac{t_0}{4})^{\frac{3}{8}}  \right) \sfed K_3(t_0, t_1) \equiv K_3
	\]
	 for all $t \in (\tfrac{t_0}{2}, t_1)$, our previous considerations directly yield  
	\[
		M_\eps \leq K_4 M_\eps^{\frac{3}{4}} + K_4 \leq \frac{3}{4}M_\eps+ \frac{K_4^4}{4} + K_4
	\] 
	with $K_4 \equiv K_4(t_0, t_1) \defs \max(K_1K_2, K_1^\frac{5}{4}K_2K_3)$ and thus 
	\[
		M_\eps \leq K_4^4 + 4K_4,
	\]
	which readily implies our desired result.
\end{proof}
\noindent
We now use the above bounds to apply standard regularity theory by Porzio and Vespri (cf.\ \cite{PorzioVespriHoelder}), which conveniently is already formulated in such a way as to work with minimal initial data regularity, to gain uniform space-time Hölder bounds for all three solution components.
\begin{lemma}\label{lemma:hoelder_1}
	Assume we are in Scenario (\ref{scenario_1}), Scenario (\ref{scenario_2}) with $C_{\mathrm{S2}}$ as in \Cref{lemma:elliptic_bound} or in Scenario (\ref{scenario_3}). 
	Then for each $t_0, t_1 \in \R$ with $t_1 > t_0 > 0$, there exist $C \equiv C(t_0, t_1) > 0$ and $\theta \equiv \theta(t_0, t_1) \in (0,1)$ such that 
	\[
		\|\ue\|_{C^{\theta, \frac{\theta}{2}}(\overline{\Omega}\times[t_0, t_1])} \leq C
	\]
	and
	\[
		\|\ve\|_{C^{\theta, \frac{\theta}{2}}(\overline{\Omega}\times[t_0, t_1])} \leq C \stext{ as well as } \|\we\|_{C^{\theta, \frac{\theta}{2}}(\overline{\Omega}\times[t_0, t_1])} \leq C
	\]
	for all $\eps \in (0,1)$.
\end{lemma}
\begin{proof}
	Reframing the first equation in (\ref{problem_simplified}) as 
	\[
		\uet - \div (a_\eps(x,t,\ue,\grad \ue)) = b(x,t,\ue, \grad\ue)
	\]
	with $a_\eps(x,t,\phi,\Phi) \defs \Phi + \ue(x, t)\grad \ze(x,t)$ and $b(x,t,\phi,\Phi) \defs 0$ for all $(x,t,\phi,\Phi) \in \Omega\times(0,\infty)\times\R\times\R^n$ makes it available to the results about Hölder regularity proven in \cite{PorzioVespriHoelder}. In fact, the bounds established in \Cref{lemma:gradz_bound} and \Cref{lemma:u_linfty_bound} are already sufficient to gain the first half of our desired result by \cite[Theorem 1.3]{PorzioVespriHoelder}.
	\\[0.5em] 
	Similarly if $\tau = 1$, the second equation in (\ref{problem}) can be rewritten as 
	\[
		\vet - \div (a(x,t,\ve,\grad\ve)) = b_\eps(x,t,\ve,\grad \ve)
	\]
	with $a(x,t,\phi, \Phi) \defs \Phi$ and $b_\eps(x,t,\phi,\Phi) \defs \alpha \ue(x,t) - \beta\ve(x,t)$ for all $(x,t,\phi,\Phi) \in \Omega\times(0,\infty)\times\R\times\R^n$, making it available to essentially the same argument as above due to the bounds in \Cref{lemma:gradz_bound} and \Cref{lemma:u_linfty_bound}. Given the similar structure of the third equation in (\ref{problem}) this also holds true for $\we$ and thus gives us the second half of our result.
	\\[0.5em]
	If $\tau = 0$, then our desired bounds for $\ve$ and $\we$ follow from standard elliptic regularity theory and the space-time Hölder bound for $\ue$ already established above (cf.\ \cite[Lemma 5.5]{HeihoffExistenceGlobalSmooth2021} for a similar argument). 
\end{proof}
\noindent
Using the parabolic regularity theory from \cite{LadyzenskajaLinearQuasilinearEquations1988} and \cite{LiebermanHolderContinuityGradient1987} now allows us to gain the final bounds necessary for our solution construction, namely higher order space-time Hölder bounds. As said regularity theory requires higher initial data regularity than we have uniformly available, we combine it with a straightforward cutoff function argument to decouple it from the initial data.
\begin{lemma}\label{lemma:hoelder_2}
	Assume we are in Scenario (\ref{scenario_1}), Scenario (\ref{scenario_2}) with $C_{\mathrm{S2}}$ as in \Cref{lemma:elliptic_bound} or in Scenario (\ref{scenario_3}).
	Then for each $t_0, t_1 \in \R$ with $t_1 > t_0 > 0$, there exist $C \equiv C(t_0, t_1) > 0$ and $\theta \equiv \theta(t_0, t_1) \in (0,1)$ such that 
	\[
		\|\ue\|_{C^{2+\theta, 1+\frac{\theta}{2}}(\overline{\Omega}\times[t_0, t_1])} \leq C	
	\]
	and 
	\[
		\|\ve\|_{C^{2+\theta, \tau+\frac{\theta}{2}}(\overline{\Omega}\times[t_0, t_1])} \leq C	\stext{ as well as } \|\we\|_{C^{2+\theta, \tau+\frac{\theta}{2}}(\overline{\Omega}\times[t_0, t_1])} \leq C
	\]
	for all $\eps \in (0,1)$.
\end{lemma}
\begin{proof}
	Fix $t_1 > t_0 > 0$. We then further fix a cutoff function $\rho \in C^\infty([\tfrac{t_0}{4}, t_1])$ with
	\[
		\rho \equiv 0 \text{ on } [\tfrac{t_0}{4}, \tfrac{t_0}{2}], \;\;\;\;
		\rho(t) \in [0,1] \text{ for all } t\in [t_0/2, t_0] \stext{ and } \rho \equiv 1 \text{ on } [t_0, t_1].
	\]
	We now begin by treating $\ve$ and $\we$. If $\tau = 1$, we first let $\tilde{\ve}(x,t) \defs \rho(t) \ve(x,t)$ for all $x\in\overline{\Omega}$, $t\in[t_0/4, t_1]$ and $\eps \in (0,1)$. These functions then are classical solutions of 
	\[
		\left\{
		\begin{aligned}
			\tilde{\ve}_t &= \laplace{\tilde{\ve}} + f_\eps(x,t) &&\text{ on } \Omega \times (\tfrac{t_0}{4}, t_1), \\
			\grad \tilde{\ve} \cdot \nu &= 0 && \text{ on } \partial\Omega\times(\tfrac{t_0}{4}, t_1),\\
			\tilde{\ve}(\cdot, \tfrac{t_0}{4}) &= 0 &&\text{ on } \Omega 
		\end{aligned}
		\right.	
	\]
	with $f_\eps(x,t) \defs [\rho'(t)-\beta\rho(t)]\ve(x,t) + \alpha\rho(t)\ue(x,t)$ for all $x\in\overline{\Omega}$, $t\in[t_0/4, t_1]$ and $\eps \in (0,1)$. As the bounds established in \Cref{lemma:hoelder_1} ensure that there exist $K_1 \equiv K_1(t_0, t_1) > 0$ and $\theta_1 \equiv \theta_1(t_0, t_1) \in (0,1)$ such that 
	\[
		\|f_\eps\|_{C^{\theta_1, \frac{\theta_1}{2}}(\overline{\Omega}\times[t_0/4, t_1])} \leq K_1
	\]
	for all $\eps \in (0,1)$, we can apply standard parabolic regularity theory from \cite[p.170 and p.320]{LadyzenskajaLinearQuasilinearEquations1988} to gain constants $K_2 \equiv K_2(t_0, t_1) > 0$ and $\theta_2 \equiv \theta_2(t_0, t_1) \in (0,1)$ such that 
	\[
		\|\tilde{\ve}\|_{C^{2+\theta_2, \tau+\frac{\theta_2}{2}}(\overline{\Omega}\times[t_0/4, t_1])} \leq K_2
	\] for all $\eps \in (0,1)$. As $\rho \equiv 1$ on $[t_0, t_1]$, this directly implies 
	\[
		\|\ve\|_{C^{2+\theta_2, \tau+\frac{\theta_2}{2}}(\overline{\Omega}\times[t_0, t_1])} \leq K_2
	\]
	for all $\eps \in (0,1)$. By essentially the same argument, we further gain $K_3 \equiv K_3(t_0, t_1) > 0$ and $\theta_3 \equiv \theta_3(t_0, t_1) \in (0,1)$ such that 
	\[
		\|\we\|_{C^{2+\theta_3, \tau+\frac{\theta_3}{2}}(\overline{\Omega}\times[t_0, t_1])} \leq K_3
	\]
	for all $\eps \in (0,1)$.
	\\[0.5em]
	If $\tau = 0$, then the corresponding bounds follow directly from standard elliptic regularity theory (cf.\ \cite{LUElliptic}) in a similar fashion to the arguments seen in \cite[Lemma 5.5]{HeihoffExistenceGlobalSmooth2021}.
	\\[0.5em]
	These arguments of course immediately also ensure the existence of constants $K_4 \equiv K_4(t_0, t_1) > 0$ and $\theta_4 \equiv \theta_4(t_0, t_1) \in (0,1)$ such that 
	\begin{equation}\label{eq:z_hoelder2_bound}
		\|\ze\|_{C^{2+\theta_4, \tau+\frac{\theta_4}{2}}(\overline{\Omega}\times[t_0, t_1])} \leq K_4		
	\end{equation}
	for all $\eps \in (0,1)$ as $\ze$ is a linear combination of $\ve$ and $\we$.
	\\[0.5em]
	Given that we have now established the second half of our result, we can focus our attention on $\ue$. We again begin by setting $\tilde{\ue}(x,t) \defs \rho(t) \ue(x,t)$ for all $x\in\overline{\Omega}$, $t\in[t_0/4, t_1]$ and $\eps \in (0,1)$. 
	Then these functions are a classical solution of the system
	\[
		\left\{
		\begin{aligned}
			\tilde{\ue}_t &= \div (A_\eps(x,t,\tilde{\ue},\grad \tilde{\ue})) + B_\eps(x,t,\tilde{\ue},\grad \tilde{\ue}) &&\text{ on } \Omega \times (\tfrac{t_0}{4}, t_1), \\
			A_\eps(x,t,\tilde{\ue},\grad \tilde{\ue}) \cdot \nu &= 0 && \text{ on } \partial\Omega\times(\tfrac{t_0}{4}, t_1),\\
			\tilde{\ue}(\cdot, \tfrac{t_0}{4}) &= 0 &&\text{ on } \Omega 
		\end{aligned}
		\right.	
	\]
	with $A_\eps(x,t,\phi,\Phi) \defs \Phi + \rho(t)\ue(x, t)\grad \ze(x,t)$ and $B_\eps(x,t,\phi,\Phi) \defs \rho'(t) \ue(x,t)$ for all $(x,t,\phi,\Phi) \in \Omega\times(0,\infty)\times\R\times\R^n$ and $\eps \in (0,1)$, which is compatible with the notation from \cite{LiebermanHolderContinuityGradient1987}. As \Cref{lemma:hoelder_1} and (\ref{eq:z_hoelder2_bound}) ensure that there exist $K_5 \equiv K_5(t_0, t_1) > 0$ and $\theta_5 \equiv \theta_5(t_0, t_1) \in (0,1)$ such that
	\[
		\|\ue\|_{L^\infty(\Omega\times(t_0/4,t_1))} \leq K_5, \;\;\;\; \|\rho\ue\grad \ze\|_{C^{\theta_5, \frac{\theta_5}{2}}(\overline{\Omega}\times[t_0/4, t_1])} \leq K_5 \stext{ and } \| B_\eps\|_{L^\infty(\Omega\times(t_0/4,t_1)\times\R\times\R^n)} \leq K_5
	\]
	after a standard time shift for all $\eps \in (0,1)$, we can apply \cite[Theorem 1.1]{LiebermanHolderContinuityGradient1987} to gain $K_6 \equiv K_6(t_0, t_1) > 0$ and $\theta_6 \equiv \theta_6(t_0, t_1) \in (0,1)$ such that 
	\begin{equation}\label{eq:grad_u_hoelder_bound}
		\|\grad \tilde{\ue}\|_{C^{\theta_6}(\overline{\Omega}\times[t_0/4, t_1])} \leq K_6	\stext{ and thus } \|\grad \ue\|_{C^{\theta_6}(\overline{\Omega}\times[t_0, t_1])} \leq K_6
	\end{equation}
	for all $\eps \in (0,1)$.
	\\[0.5em]
	Slightly reframing the above system to make it available to the results presented in \cite{LadyzenskajaLinearQuasilinearEquations1988}, $\tilde{\ue}$ is also a classical solution to
	\[
		\left\{
		\begin{aligned}
			\tilde{\ue}_t &= \laplace{\tilde{\ue}} + f_\eps(x,t) &&\text{ on } \Omega \times (\tfrac{t_0}{4}, t_1) \\
			\grad \tilde{\ue} \cdot \nu &= 0 && \text{ on } \partial\Omega\times(\tfrac{t_0}{4}, t_1)\\
			\tilde{\ue}(\cdot, \tfrac{t_0}{4}) &= 0 &&\text{ on } \Omega 
		\end{aligned}
		\right.	
	\]
	with $f_\eps(x,t) \defs \rho(t)\grad \ue(x,t) \cdot \grad \ze(x,t) +\rho(t)\ue(x,t)\laplace\ze(x,t) + \rho'(t)\ue(x,t)$ for all $x\in\overline{\Omega}$, $t\in[t_0/4, t_1]$ and $\eps\in(0,1)$. Due to the bounds established in \Cref{lemma:hoelder_1}, (\ref{eq:z_hoelder2_bound}) as well as (\ref{eq:grad_u_hoelder_bound}), there exist $K_7 \equiv K_7(t_0, t_1) > 0$ and $\theta_7 \equiv \theta_7(t_0, t_1) \in (0,1)$ such that 
	\[
		\|f_\eps\|_{C^{\theta_7, \frac{\theta_7}{2}}(\overline{\Omega}\times[t_0/4, t_1])} \leq K_7 
	\]
	after a standard time shift for all $\eps \in (0,1)$. This then allows us to apply the same regularity theory from \cite[p.170 and p.320]{LadyzenskajaLinearQuasilinearEquations1988} as before to gain $K_8 \equiv K_8(t_0, t_1) > 0$ and $\theta_8 \equiv \theta_8(t_0, t_1) \in (0,1)$ such that 
	\[
		\|\tilde{\ue}\|_{C^{2+\theta_8, 1+\frac{\theta_8}{2}}(\overline{\Omega}\times[t_0/4, t_1])} \leq K_8 \stext{ and thus } \|\ue\|_{C^{2+\theta_8, 1+\frac{\theta_8}{2}}(\overline{\Omega}\times[t_0, t_1])} \leq K_8
	\]
	for all $\eps \in (0,1)$, which completes the proof with $\theta \equiv \theta(t_0, t_1) \defs \min(\theta_2, \theta_3, \theta_8)$ and an appropriate constant $C$. 
\end{proof}\noindent
Using these bounds, we now construct our desired solution candidates as follows:
\begin{lemma}\label{lemma:construction}
Assume we are in Scenario (\ref{scenario_1}), Scenario (\ref{scenario_2}) with $C_{\mathrm{S2}}$ as in \Cref{lemma:elliptic_bound} or in Scenario (\ref{scenario_3}). Then there exist nonnegative functions $u\in C^{2,1}(\overline{\Omega}\times(0,\infty))$ and $v,w \in C^{2,\tau}(\overline{\Omega}\times(0,\infty))$ as well as a null sequence $(\eps_j)_{j\in\N}$ such that 
\begin{align}
	\ue &\rightarrow u &&\text{ in } C^{2,1}(\overline{\Omega}\times[t_0, t_1]) \label{eq:u_convergence}\\
	\ve &\rightarrow v &&\text{ in } C^{2,\tau}(\overline{\Omega}\times[t_0, t_1]) \label{eq:v_convergence}\\\
	\we &\rightarrow w &&\text{ in } C^{2,\tau}(\overline{\Omega}\times[t_0, t_1])\label{eq:w_convergence}\
\end{align}
for all $t_1 > t_0 > 0$ as $\eps = \eps_j \searrow 0$ and such that $(u, v, w)$ is a classical solution to (\ref{problem}).
\end{lemma}
\begin{proof}
	Given the bounds established in \Cref{lemma:hoelder_2}, the convergence properties (\ref{eq:u_convergence}) to (\ref{eq:w_convergence}) as well as the existence of nonnegative functions $u$, $v$, $w$ of appropriate regularity are an immediate consequence of the compact embedding properties of Hölder spaces combined with a straightforward diagonal sequence argument. As our approximate solutions $\ue$, $\ve$, $\we$ are already classical solutions to (\ref{problem}) and given the now established fairly strong convergence properties in (\ref{eq:u_convergence}) to (\ref{eq:w_convergence}), our desired solution properties immediately transfer from the approximate solutions to their limits $u$, $v$ and $w$. This completes the proof.
\end{proof}

\section{Continuity at $t = 0$}

Having now constructed solution candidates, which already solve (\ref{problem}) classically, the only thing that remains to be shown is that said candidates are connected to the initial data in the way outlined in (\ref{eq:u_convergence}) to  (\ref{eq:w_convergence}). We first focus on (\ref{eq:u_convergence}) as it is not only relevant in all of our scenarios but also the most challenging of the continuity properties. 
In fact, the key to deriving this property is to first show that our approximate solutions are continuous at $t=0$ in an appropriate and $\eps$ independent fashion. To do this, we first show that a space-time integral of a quantity connected to the taxis becomes uniformly small as $t$ goes to zero.

\begin{lemma}\label{lemma:uniform_u_continuity_help}
Assume we are in Scenario (\ref{scenario_1}), Scenario (\ref{scenario_2}) with $C_{\mathrm{S2}}$ as in \Cref{lemma:elliptic_bound} or in Scenario (\ref{scenario_3}). Then there exist $C > 0$ and $\theta > 0$ such that 
\[
	\int_0^t \| \ue(\cdot, s) \grad \ze(\cdot, s)\|_\L{1} \d s \leq C t^\theta
\]
for all $t\in(0,1)$ and $\eps \in (0,1)$.
\end{lemma}
\begin{proof}
We begin by considering Scenario (\ref{scenario_1}) and more specifically we start by treating the simpler case of $\zeta$ being equal to zero. Here, we can use the variation-of-constants representation of the second equation in (\ref{problem_simplified}), smoothing properties of the Neumann heat semigroup (cf.\ \cite[Lemma 1.3 (ii)]{WinklerAggregationVsGlobal2010}), (\ref{eq:approx_vw0_properties}) combined with the Sobolev inequality and \Cref{corollary:v_l2} to find $K_1 > 0$ such that 
\begin{align*}
	\|\grad \ze(\cdot, t)\|_\L{\infty} &= \left\| \, \grad e^{t(\laplace - \delta)}\ze(\cdot, 0) + \sigma\int_0^t \grad e^{(t-s)(\laplace - \delta)} \ve(\cdot, s) \d s \,\right\|_{\L{\infty}} \\
	&\leq K_1 (1 + t^{-\frac{1}{2} - \frac{2-r}{2r}}) + |\sigma| K_1 \int_0^t (1 + (t-s)^{-\frac{1}{2} - \frac{2-r}{2r}}) \d s \\ 
	&\leq 2K_1 t^{-\frac{1}{r}} + 2|\sigma| K_1 \int_0^t (t-s)^{-\frac{1}{r}} \d s \\ 
	&= 2K_1 t^{-\frac{1}{r}} +  2\tfrac{r}{r-1}|\sigma| K_1 t^\frac{r-1}{r} \leq 2(1+ |\sigma|)\tfrac{r}{r-1}K_1t^{-\frac{1}{r}}
\end{align*}
for all $t\in(0,1)$ and $\eps \in (0,1)$ as $\zeta = 0$. Due to the mass conservation property (\ref{eq:mass_conservation}), this directly implies 
\begin{align*}
	\int_0^t\|\ue(\cdot, s) \grad \ze(\cdot, s)\|_\L{1} \d s &\leq m \int_0^t\|\grad \ze(\cdot,s)\|_\L{\infty} \d s \\ 
	&\leq 2m(1+ |\sigma|)\tfrac{r}{r-1}K_1 \int_0^t s^{-\frac{1}{r}} \d s \\	 
	&= 2m(1+ |\sigma|)(\tfrac{r}{r-1})^2 K_1 t^\frac{r-1}{r}
\end{align*}
for all $t\in(0,1)$ and $\eps \in (0,1)$, which sufficiently addresses this case as $r > 1$.
\\[0.5em]
Let us now consider the more subtle case of $\zeta$ being positive in Scenario (\ref{scenario_1}). Here, we first fix $\lambda \in (0, \frac{2}{3})$ as in \Cref{lemma:energy} and an appropriately small constant $\theta \in (0,1)$ such that
\[
	-\frac{3}{2}\lambda - \theta > -1 + \theta \stext{ as well as }-\frac{2}{3}\lambda - \frac{1}{3}\theta > -1 + \theta.	
\]
Using Young's inequality, we then see that 
\begin{equation}\label{eq:t0_continuity_estimate_parab}
	\|\ue(\cdot, t) \grad \ze(\cdot, t)\|_\L{1} \leq t^{-\frac{2}{3}\lambda-\frac{1}{3}\theta}\int_\Omega \ue^{\frac{4}{3}} + t^{2\lambda+\theta}\int_\Omega |\grad \ze|^4
\end{equation}
for all $t\in(0,1)$ and $\eps \in (0,1)$.
Due to \Cref{corollary:energy}, we already know that there exists $K_2 > 0$ such that 
\begin{equation}\label{eq:t0_continuity_estimate_gradv}
	\int_0^t s^{2\lambda+\theta}\int_\Omega |\grad \ze(x, s)|^4 \d x \d s \leq t^\theta\int_0^t s^{2\lambda}\int_\Omega |\grad \ze(x, s)|^4 \d x \d s \leq K_2 t^{\theta} 
\end{equation}
for all $t\in(0,1)$ and $\eps \in (0,1)$. As such, we can now focus on the remaining term involving $\ue$ in (\ref{eq:t0_continuity_estimate_parab}). To this end, we use the Gagliardo--Nirenberg inequality as well as the mass conservation property (\ref{eq:mass_conservation}) to gain $K_3 > 0$ such that 
\begin{align*}
	t^{-\frac{2}{3}\lambda-\frac{1}{3}\theta}\int_\Omega \ue^{\frac{4}{3}} &= t^{-\frac{2}{3}\lambda-\frac{1}{3}\theta} \|\ue^\frac{1}{2}\|^\frac{8}{3}_\L{\frac{8}{3}} \leq K_3 t^{-\frac{2}{3}\lambda-\frac{1}{3}\theta}\left(\int_\Omega \frac{|\grad \ue|^2}{\ue} \right)^\frac{1}{3} + K_3t^{-\frac{2}{3}\lambda-\frac{1}{3}\theta} \\
	&= K_3 t^{-\lambda-\frac{2}{3}\theta}\left( t^{\lambda + \theta}\int_\Omega \frac{|\grad \ue|^2}{\ue} \right)^\frac{1}{3} + K_3t^{-\frac{2}{3}\lambda-\frac{1}{3}\theta} \\
	&\leq K_3 t^{\lambda + \theta}\int_\Omega \frac{|\grad \ue|^2}{\ue} + K_3 t^{-\frac{3}{2}\lambda - \theta} + K_3t^{-\frac{2}{3}\lambda-\frac{1}{3}\theta} \\
	&\leq K_3 t^{\lambda + \theta}\int_\Omega \frac{|\grad \ue|^2}{\ue} + 2K_3t^{-1+\theta}
\end{align*}
for all $t\in(0,1)$ and $\eps \in (0,1)$, where the last step is facilitated by our choice of $\theta$. Thus by application of \Cref{lemma:energy}, we gain $K_4 > 0$ such that
\begin{align*}
	\int_0^t s^{-\frac{2}{3}\lambda-\frac{1}{3}\theta}\int_\Omega \ue^{\frac{4}{3}}(x,s) \d x \d s 
	&\leq K_3 \int_0^t s^{\lambda + \theta}\int_\Omega \frac{|\grad \ue(x,s)|^2}{\ue(x,s)} \d x \d s + 2K_3\int_0^t s^{-1+\theta} \d s \\
	&\leq K_3 t^\theta \int_0^t s^{\lambda}\int_\Omega \frac{|\grad \ue(x,s)|^2}{\ue(x,s)} \d x \d s + 2K_3\int_0^t s^{-1+\theta} \d s \\
	&\leq K_3 K_4 t^{\theta} + \frac{2K_3}{\theta} t^\theta 
\end{align*}
for all $t\in(0,1)$ and $\eps \in (0,1)$ as $\zeta > 0$. Combining this with the similar estimate (\ref{eq:t0_continuity_estimate_gradv}) and the estimate (\ref{eq:t0_continuity_estimate_parab}) then directly yields our desired result for Scenario (\ref{scenario_1}) with $\zeta > 0$.
\\[0.5em]
We next look at Scenario (\ref{scenario_2}). Using \Cref{lemma:elliptic_bound} and \Cref{lemma:baseline_elliptic}, we here immediately gain $K_5 > 0$ such that 
\begin{align*}
	\int_0^t \| \ue(\cdot, s) \grad \ze(\cdot, s)\|_\L{1} \d s &\leq \int_0^t \| \ue(\cdot, s) \|_\L{\frac{5}{2}} \|\grad \ze(\cdot, s)\|_\L{\frac{5}{3}} \d s \leq K_5\int_0^t s^{-\frac{3}{2}\cdot\frac{2}{5}} \d s = \frac{5K_5}{2} \, t^\frac{2}{5}
\end{align*}
for all $t \in (0,1)$ and $\eps \in (0,1)$, which completes the proof for this case.
\\[0.5em]
Lastly, we will handle Scenario (\ref{scenario_3}) as follows: According to \Cref{lemma:elliptic_bound}, there exists a uniform bound for $\ue$ in $L^\uparam(\Omega)$ on the time interval $(0,1)$ and thus by elliptic regularity theory there further exist uniform bounds for $\ve$, $\we$ and therefore $\ze$ in $W^{2,\uparam}(\Omega)$ on the same time interval. Using the Sobolev inequality, this means we can fix $K_6 > 0$ such that 
\[
	\|\grad \ze(\cdot,t)\|_\L{\frac{3\uparam}{3-\uparam}} \leq K_6
\]
for all $t\in (0,1)$ and $\eps \in (0,1)$. Having established this, we can then use \Cref{lemma:elliptic_bound} to find $K_7 > 0$ such that
\begin{align*}
	\int_0^t \| \ue(\cdot, s) \grad \ze(\cdot, s)\|_\L{1} \d s 
	&\leq \int_0^t \| \ue(\cdot, s) \|_\L{\frac{3\uparam}{4\uparam-3}} \|\grad \ze(\cdot, s)\|_\L{\frac{3\uparam}{3-\uparam}} \d s \\
	&\leq K_6K_7 \int_0^t s^{-\frac{3-\uparam}{2\uparam}} \d s = \frac{2\uparam K_5 K_6}{3\uparam-3} \, t^{\frac{3\uparam-3}{2\uparam}}
\end{align*}
for all $t\in (0,1)$ and $\eps \in (0,1)$. As $\frac{3\uparam-3}{2\uparam} > 0$, this in fact completes the proof.
\end{proof}
\noindent
Using the fundamental theorem of calculus and the first equation in (\ref{problem_simplified}), we now derive the following uniform continuity property for our approximate solutions from the above result.

\begin{lemma}\label{lemma:t0_ue_continuity}
Assume we are in Scenario (\ref{scenario_1}), Scenario (\ref{scenario_2}) with $C_{\mathrm{S2}}$ as in \Cref{lemma:elliptic_bound} or in Scenario (\ref{scenario_3}). Then for each $\phi \in C^0(\overline{\Omega})$ and $\eta > 0$, there exists $t_0 \equiv t_0(\phi, \eta) \in (0,1)$ such that 
\[
	\left| \int_\Omega \ue(\cdot, t) \phi - \int_\Omega u_{0, \eps}\phi  \right| \leq \eta
\]
for all $t \in (0,t_0)$ and $\eps \in (0,1)$.
\end{lemma}
\begin{proof}
	Let $\eta > 0$ and $\phi \in C^0(\overline{\Omega})$ be fixed but arbitrary.
	\\[0.5em]
	As $C^2(\overline{\Omega})$ with Neumann zero is dense in $C^0(\overline{\Omega})$, we can further fix  $\psi \in C^2(\overline{\Omega})$ with $\grad \psi \cdot \nu = 0$ on $\partial \Omega$ and 
	\begin{equation}\label{eq:psi_property}
		\|\phi - \psi\|_\L{\infty} \leq \frac{\eta}{3m}.
	\end{equation}
	Using the fundamental theorem of calculus, the first equation in (\ref{problem_simplified}), partial integration and \Cref{lemma:uniform_u_continuity_help}, we can then fix $K_1 > 0$ and $\theta > 0$ such that 
	\begin{align*}
		\left| \int_\Omega \ue(\cdot, t)\psi - \int_\Omega u_{0, \eps} \psi \right| &= \left| \int_0^t \int_\Omega u_{\eps t}\psi \right|
		= \left| \int_0^t \int_\Omega \laplace \ue \psi + \int_0^t\int_\Omega \div (\ue \grad \ze) \psi \right| \\
		&= \left| \int_0^t \int_\Omega \ue \laplace \psi - \int_0^t\int_\Omega (\ue \grad \ze) \cdot \grad\psi \right| \\
		&\leq m\|\laplace \psi \|_\L{\infty}t + \|\grad \psi\|_\L{\infty}\int_0^t \|\ue(\cdot, s) \grad \ze(\cdot, s)\|_\L{1} \d s \\
		&\leq m\|\laplace \psi \|_\L{\infty}t + K_1 \|\grad \psi\|_\L{\infty} t^\theta 
	\end{align*}
	for all $t\in(0,1)$ and $\eps \in (0,1)$. Thus we can find a sufficiently small $t_0 > 0$ such that 
	\[
		\left| \int_\Omega \ue(\cdot, t)\psi - \int_\Omega u_{\eps, 0} \psi \right| \leq \frac{\eta}{3}
	\]
	for all $t\in(0,t_0)$ and $\eps \in (0,1)$.
	\\[0.5em]
	From this as well as (\ref{eq:psi_property}), we can further conclude that 
	\begin{align*}
		\left| \int_\Omega \ue(\cdot, t) \phi - \int_\Omega u_{0, \eps}\phi \right| &\leq \left| \int_\Omega \ue(\cdot, t) \phi - \int_\Omega \ue(\cdot, t) \psi \right| \\
		&\hphantom{=}+ \left|\int_\Omega \ue(\cdot, t) \psi - \int_\Omega u_{0,\eps} \psi \right| \\
		&\hphantom{=}+\left|\int_\Omega u_{0,\eps} \psi - \int_\Omega u_{0,\eps} \phi \right| \\
		&\leq m \frac{\eta}{3m} + \frac{\eta}{3} + m \frac{\eta}{3m} = \eta
	\end{align*}
	for all $t\in(0,t_0)$ and $\eps \in (0,1)$, which is exactly our desired result.
\end{proof}
\noindent
We now only need to argue that the above property survives the limit process undertaken in \Cref{lemma:construction} to gain the first of the continuity properties at $t=0$ for our solution candidates.
\begin{lemma}\label{lemma:t0_u_continuity}
	Assume we are in Scenario (\ref{scenario_1}), Scenario (\ref{scenario_2}) with $C_{\mathrm{S2}}$ as in \Cref{lemma:elliptic_bound} or in Scenario (\ref{scenario_3}). If $u$ is the function constructed in \Cref{lemma:construction}, then 
	\[
		u(\cdot, t) \rightarrow u_0 \stext{ in } \Mp  	
	\]
	as $t\searrow 0$, where $\Mp$ is the set of positive Radon measures with the vague topology and we interpret the functions $u(\cdot, t)$, $t > 0$, as the positive Radon measures $ u(x,t)\mathrm{d} x$ with $\mathrm{d} x$ being the standard Lebesgue measure on $\Omega$.
\end{lemma}
\begin{proof}
	Let $\eta > 0$ and $\phi \in C^0(\overline{\Omega})$ be fixed but arbitrary.
	\\[0.5em]
	According to \Cref{lemma:t0_ue_continuity}, we can then fix $t_0 > 0$ such that 
	\[
		\left| \int_\Omega \ue(\cdot, t) \phi - \int_\Omega u_{0,\eps}\phi  \right| \leq \frac{\eta}{3}
	\]
	for all $t\in(0,t_0)$ and $\eps \in (0,1)$.
	\\[0.5em]
	Using (\ref{eq:approx_u0_properties}) as well as \Cref{lemma:construction}, we can further, for each $t\in(0,t_0)$, find $\eps(t) \in (0,1)$, such that  
	\[
		\left|\int_\Omega u_{0, \eps(t)}\phi - \int_{\overline{\Omega}} \phi \d u_0 \right| \leq \frac{\eta}{3} \stext{ and } \left| \int_\Omega u(\cdot, t)\phi - \int_\Omega u_{\eps(t)}(\cdot, t)\phi\right| \leq \frac{\eta}{3}.
	\]
    Combining the above estimates then yields
	\begin{align*}
		\left| \int_\Omega u(\cdot, t)\phi - \int_{\overline{\Omega}} \phi \d u_0  \right| &\leq \left| \int_\Omega u(\cdot, t)\phi - \int_\Omega u_{\eps(t)}(\cdot, t)\phi \right| \\
		&\hphantom{=}+ \left| \int_\Omega u_{\eps(t)}(\cdot, t)\phi - \int_\Omega u_{0,\eps(t)}\phi  \right| \\
		&\hphantom{=}+ \left|\int_\Omega u_{0,\eps(t)}\phi - \int_{\overline{\Omega}} \phi \d u_0  \right| \\
		&\leq \frac{\eta}{3} + \frac{\eta}{3} + \frac{\eta}{3} = \eta 
	\end{align*}
	for all $t\in(0,t_0)$. This completes the proof.
\end{proof}
\noindent
As our last significant step of this paper, we will now treat the remaining continuity properties at $t=0$ for the parabolic-parabolic case. Similar to our prior arguments in this section, we will do this by first showing that the approximate solutions are uniformly continuous at $t=0$ by way of semigroup methods and then translate this continuity property to our limit solutions.

\begin{lemma}\label{lemma:t0_vw_continuity}
	Assume we are in Scenario (\ref{scenario_1}). If $v$, $w$ are the functions constructed in \Cref{lemma:construction}, then 
	\[
		v(\cdot, t) \rightarrow v_0 \stext{ and } w(\cdot, t) \rightarrow w_0 \stext{ in } W^{1,r}(\Omega) 	
	\]
	as $t\searrow 0$ with $r\in(\frac{6}{5}, 2)$ as fixed at the beginning of \Cref{section:apriori}.	
\end{lemma}
\begin{proof}
	Let $\eta > 0$ be fixed but arbitrary.
	\\[0.5em]
	Using the variation-of-constants representation of the second equation in (\ref{problem}), we first note that 
	\begin{align*}
		\|\ve(\cdot, t) - v_{0,\eps}\|_{W^{1,r}(\Omega)} &\leq  \|e^{t(\laplace-\beta)}v_{0,\eps} - v_{0,\eps}\|_{W^{1,r}(\Omega)} + \alpha\int_0^t\|e^{(t-s)(\laplace - \beta)}\ue(\cdot, s)\|_{W^{1,r}(\Omega)} \d s 
	\end{align*}
	for all $t \in (0,1)$ and $\eps \in (0,1)$.
	\\[0.5em]
	To treat the first summand, we recall our definition of $v_{0,\eps}$ in (\ref{eq:v0_w0_definition}) and standard properties of the Neumann heat semigroup to gain $K_1 > 0$ such that 
	\begin{align*}
		\|e^{t(\laplace-\beta)}v_{0,\eps} - v_{0,\eps}\|_{W^{1,r}(\Omega)}
		&= \|e^{t(\laplace-\beta)}e^{\eps(\laplace - \beta)}v_0 - e^{\eps(\laplace - \beta)}v_0\|_{W^{1,r}(\Omega)}  
		= \|e^{\eps(\laplace - \beta)} (e^{t(\laplace-\beta)}v_0 - v_0 )\|_{W^{1,r}(\Omega)}  \\
		&\leq K_1 \| e^{t(\laplace-\beta)}v_0 - v_0  \|_{W^{1,r}(\Omega)} 
	\end{align*}
	for all $t\in(0,1)$ and $\eps \in (0,1)$. Due to further continuity properties of the Neumann heat semigroup and the fact that $v_0 \in W^{1,r}(\Omega)$, this then allows us to fix $t_1 > 0$ such that 
	\[
		\|e^{t(\laplace-\beta)}v_{0,\eps} - v_{0,\eps}\|_{W^{1,r}(\Omega)} \leq \frac{\eta}{6}
	\]
	for all $t\in(0, t_1)$ and $\eps \in (0,1)$.
	\\[0.5em]
	Regarding the second summand, we begin by employing the smoothing properties of the Neumann heat semigroup (cf.\ \cite[Lemma 1.3]{WinklerAggregationVsGlobal2010}) to gain $K_2 > 0$ such that 
	\[
		\int_0^t\|e^{(t-s)(\laplace - \beta)}\ue(\cdot, s)\|_{W^{1,r}(\Omega)} \d s \leq K_2 \int_0^t ( 1+ (t-s)^{-\frac{3}{2} + \frac{1}{r}}) \|\ue(\cdot, s)\|_\L{1} \d s = mK_2t + \tfrac{m K_2}{\frac{1}{r} - \frac{1}{2}} \, t^{\frac{1}{r} - \frac{1}{2}}  
	\]
	for all $t \in (0,1)$ and $\eps \in (0,1)$. As $r \in (\frac{6}{5},2)$, we can therefore fix $t_2 > 0$ such that 
	\[
		\alpha\int_0^t\|e^{(t-s)(\laplace - \beta)}\ue(\cdot, s)\|_{W^{1,r}(\Omega)} \leq \frac{\eta}{6}
	\]
	for all $t\in (0,t_2)$ and $\eps \in (0,1)$. Thus 
	\begin{equation}\label{eq:ve_uniform_continuity}
		\|\ve(\cdot, t) - v_{0,\eps}\|_{W^{1,r}(\Omega)} \leq \frac{\eta}{3}
	\end{equation}
	for all $t \in (0,t_0)$ and $\eps \in (0,1)$ with $t_0 \defs \min(t_1, t_2)$.
	\\[0.5em]
	Due to (\ref{eq:approx_vw0_properties}) and \Cref{lemma:construction}, we can, for each $t \in (0,1)$, further find $\eps(t) \in (0,1)$ such that 
	\[
		\|v_{0,\eps(t)} - v_0\|_{W^{1,r}(\Omega)} \leq \frac{\eta}{3} \stext{ and } 	\|v(\cdot, t) - v_{\eps(t)}(\cdot, t)\|_{W^{1,r}(\Omega)} \leq \frac{\eta}{3}.
	\]
 	Using these estimates as well as (\ref{eq:ve_uniform_continuity}), we then observe that 
	\begin{align*}
		\|v(\cdot, t) - v_0\|_{W^{1,r}(\Omega)} &\leq \|v(\cdot, t) - v_{\eps(t)}(\cdot, t)\|_{W^{1,r}(\Omega)} \\
		&\hphantom{=}+ \|v_{\eps(t)}(\cdot, t) - v_{0,\eps(t)}\|_{W^{1,r}(\Omega)} \\
		&\hphantom{=}+ \|v_{0,\eps(t)} - v_0\|_{W^{1,r}(\Omega)} \\
		&\leq \frac{\eta}{3} + \frac{\eta}{3} + \frac{\eta}{3} = \eta
	\end{align*}
	for all $t\in(0,t_0)$. Therefore, we have proven our desired result for $v$. Essentially the same argument applied to $w$ then completes the proof.
\end{proof}
\noindent
As we have already given all the necessary arguments in previous lemmata, the proof of our main theorem can now be given quite succinctly.
\begin{proof}[Proof of \Cref{theorem:main}]
	Let $u$, $v$, $w$ be the functions constructed in \Cref{lemma:construction}. By construction, said functions are nonnegative classical solutions to (\ref{problem}) of appropriate regularity and further, by \Cref{lemma:t0_u_continuity} and \Cref{lemma:t0_vw_continuity}, they have our desired continuity properties at $t = 0$. 
\end{proof}

\section*{Acknowledgment} The author acknowledges support of the \emph{Deutsche Forschungsgemeinschaft} in the context of the project \emph{Fine structures in interpolation inequalities and application to parabolic problems}, project number 462888149.

\end{document}